\documentclass[11pt,a4paper]{article}
\usepackage{verbatim} 
\usepackage{amsmath,amsthm,amssymb}
\usepackage{graphicx}
\usepackage{subfigure}
\usepackage{multirow}
\usepackage{epstopdf}

\usepackage{xcolor}
 
\newtheorem{theorem}{Theorem}[section]
\newtheorem{lemma}{Lemma}[section]

\newtheorem{remark}{Remark}[section]
\newcommand{\se}{\setcounter{equation}{0}}
\newcommand{\cT}{\mathcal{T}}

\newcommand{\cL}{A}

\newcommand{\atwo}{{\beta}}
\newcommand{\aone}{{\alpha}}

\newcommand{\norm}[1]{{\left\vert\kern-0.25ex\left\vert\kern-0.25ex\left\vert #1 
    \right\vert\kern-0.25ex\right\vert\kern-0.25ex\right\vert}}

\title{Galerkin FEM for a time-fractional Oldroyd-B fluid problem} 
\author{Mariam Al-Maskari\thanks{Email: s70036@student.squ.edu.om} \quad and\quad 
Samir Karaa\thanks{Email: skaraa@squ.edu.om. This research is supported by The Research Council of Oman under grant ORG/CBS/15/001.} \\
\\
FracDiff Research Group, Department of Mathematics\\ Sultan Qaboos University, 
Al-Khod 123, Muscat, Oman
}
\begin{document}
\date{}
\maketitle
\begin{abstract}
{
We consider the numerical approximation of a  generalized fractional Oldroyd-B fluid problem involving two Riemann-Liouville fractional derivatives in time. We  establish regularity results for the exact solution
which play an important role in the error analysis. A semidiscrete scheme based on the piecewise  linear Galerkin finite element method in space is  analyzed and  optimal with respect to the data regularity error estimates are established.
Further, two fully discrete schemes based on convolution quadrature in time generated by the backward Euler and the second-order backward difference methods are investigated and  related error estimates for smooth and nonsmooth data are derived. Numerical experiments are performed with different values of the problem parameters to illustrate the efficiency of the method and confirm the theoretical results.
}
\end{abstract}



\section{Introduction}
\se 
Let $\Omega$ be a bounded convex polygonal domain in $\mathbb{R}^2$  with
a boundary $\partial \Omega$ and let  $T>0$ be a fixed time. 
We consider the initial boundary-value problem for the following time-fractional Oldroyd-B fluid equation
\begin{subequations}\label{main}
  \begin{align}
  &(1+a \partial_t^\aone)u_t(x,t) =\mu(1 +b \partial_t^\atwo) \Delta u(x,t)+f(x,t) \quad\mbox{ in }\Omega\times (0,T], \label{a1}\ \intertext{with a homogeneous Dirichlet boundary condition}
   &u(x,t)= 0 \quad\mbox{ on }\partial\Omega\times (0,T],   \label{a2}\     
 \intertext{and initial conditions}
   &  u(x,0)=v(x),\quad (I^{1-\alpha}u_t)(x,0)=0 \quad\mbox{ in }\Omega, \label{a3}
  \end{align}
\end{subequations}
where $f$ and $v$ are given functions,  the parameters $\aone,\atwo\in(0,1)$, and $\mu$, $a$ and $b$  are positive constants. In \eqref{a1},  $u_t$ denotes the partial derivative of $u$ with respect to time, and   $\partial_t^\aone$ is the Riemann-Liouville  fractional derivative in time defined for $0<\alpha<1$ by
\begin{equation*} 
\partial_t^{\alpha} \varphi(t):= \frac{d }{d t}I^{1-\alpha} \varphi(t):= \frac{d }{d t}\int_0^t\omega_{1-\alpha}(t-s)\varphi(s)\,ds\quad\text{with} \quad
\omega_{\alpha}(t):=\frac{t^{\alpha-1}}{\Gamma(\alpha)},
\end{equation*}
where $\Gamma(\cdot)$ is the Gamma function. 
The second initial condition in \eqref{a3} is an appropriate condition for the model  equation \eqref{a1}, ensuring a unique solvability. This condition  replaces  the standard strong initial condition $u_t(x,0)=0$.
Indeed, one can show for instance that the solution $u$ satisfies (see Theorem 2.1)
$$u(t)=O(t^{\alpha-\beta+1})\quad \mbox{and}\quad   u_t(t)=O(t^{\alpha-\beta})\quad \mbox{as}\;  t\to 0.$$
 This indicates that when $\alpha<\beta$, the time derivative $u_t$ has a singularity at $t=0$,  and as a consequence,  the  condition $u_t(x,0)=0$ is not satisfied. However, we see that
 $$
 (I^{1-\alpha}u_t)(t)=O(t^{1-\beta})\quad \mbox{as}\;  t\to 0.
 $$
 Since $0<\beta<1$, we have $(I^{1-\alpha}u_t)(0)=0$, showing that our initial condition is valid for all $\alpha,\beta\in (0,1)$. We notice that similar arguments are presented in \cite{B2}.   


The time-fractional problem \eqref{main} has received considerable attention in recent years due to its capacity of modeling a large class of  fluids for different values of $a$ and $b$. 
For example,   when $a=0$ and $b>0$, \eqref{main}  describes a Rayleigh-Stokes problem for a generalized fractional second-grade fluid \cite{B6,EJLZ2016,pcB10,pcB9,B3,14, pcB11,23}. This  plays an important role in investigating the behavior of some non-Newtonian fluids. The problem with $b=0$ and $a>0$ describes a generalized fractional Maxwell model \cite{M.Abdullah,Biomath,pcB1,maxwell,maxwell1}, while the case with $a=b=0$ corresponds to classical Newtonian fluids (classical diffusion). 

Analytical studies for different versions of problem \eqref{main}, in one and two spatial dimensions,  
have been presented by several authors, see, for instance \cite{B6,pcB1,pcB12,pcB13}. In most of these works, the solution is represented in a form of a series with the help of Fourier and Laplace transforms.
For the general case  ($a>0$ and $b>0$), the solution was derived using double Fourier sine series in \cite{pcB3,pcB4}. In \cite{pcB2}, the solution was represented in an integral form in terms of the Mittag-Leffler function.  The results are  formal as  the convergence of the series and the regularity of the solutions have not been considered.  

The numerical approximation of the model \eqref{main} has also attracted the attention of many authors.
The case with  $a=0$ (generalized second-grade fluid) has, in particular, been studied extensively in the literature, see \cite{EJLZ2016,pcB10,pcB9,B3,14, pcB11,23}. 
In \cite{pcB10} and \cite{pcB9}, implicit and explicit finite difference schemes have been examined, for one and two-dimensional problems, and Fourier analysis has been conducted to analyze the stability and convergence of the methods. In \cite{23} and \cite{B3}, an implicit numerical approximation method is developed by transforming the problem into an integral equation. In \cite{14}, the authors have investigated a numerical scheme derived by the reproducing kernel technique, and in \cite{pcB11} a compact finite difference method and a radial basis function method was proposed.  
 The convergence analysis in most of these studies requires the solution $u$ of problem 
\eqref{main} to be sufficiently regular, including at $t=0$, which is not practically the case. Most recently in \cite{EJLZ2016},   Jin et al. have considered the numerical approximation of the solution of the homogeneous problem \eqref{main} with $a=0$ by piecewise linear finite elements in space and convolution quadrature in time, and have derived optimal error estimates with respect to the solution smoothness, expressed through the initial data $v$.  

Numerical studies for the general case when $a\neq 0$ in  \eqref{main}   are still limited. 
In \cite{B1}, the authors have derived explicit and implicit schemes in time based on the Gr\"{u}nwald-Letnikov approximation of  fractional derivatives, and performed numerical experiments to investigate the behavior of the solution. In \cite{B2},  the problem is reformulated as a Volterra integral equation,   where its kernel is represented in terms of Mittag-Leffler functions.  
 A special attention has been given to the behavior of the solution. In \cite{B5}, finite difference schemes  of second- and  fourth-order accuracy in space are  proposed with a second-order convolution quadrature in time.  An ADI algorithm was then developed for the computation of the numerical solution. In these studies, the theoretical analysis of the error is not presented.
In \cite{pcB5}, the authors have used a  standard Galerkin finite element procedure in space and the famous $L1$-scheme in time to approximate the solution of \eqref{a1}. They also obtained error estimates under high regularity assumptions on the exact solution.  It is worth mentioning that a time-fractional differential equation involving two Riemann-Liouville fractional derivatives was investigated in \cite{ZXW-2017,S.karaa}. 
Similar to \eqref{a1} with $a=0$, the equation was considered with only one initial condition $u(x,0)=v(x)$.

The aim of this paper is to develop a Galerkin FEM for problem \eqref{main} and derive optimal with respect to data regularity error estimates. Our analysis is based on exploiting Laplace transform tools with semigroup type properties of the FE solution operator. Further, we investigate two fully discrete schemes for the semidiscrete FE problem 
based on convolution quadrature in time generated by the backward Euler and 
the second-order backward difference methods. Error estimates with respect to the data regularity 
are established, see  Theorems \ref{thm:BE} and \ref{thm:SBD}. Compared to \cite{EJLZ2016}, where $a$ and $f$ in \eqref{a1} are both zeros, we have derived error estimates in the $H^1(\Omega)$-norm as well as in the $L^\infty(\Omega)$-norm. We have further analyzed the inhomogeneous problem and obtained new error estimates.

The  paper is organized as follows: In Section 2,  we present the solution theory of the mathematical model  \eqref{main} and derive properties of the solution operator, which  will play an important role in our subsequent error analysis. In Section \ref{sec:FE}, we introduce the semidiscrete piecewise linear FE scheme and derive optimal error estimates for  the homogeneous  and inhomogeneous problems. 
In Section \ref{sec:discrete}, two fully discrete schemes  based on convolution quadrature 
in time are considered.
Finally, in Section \ref{sec:NE}, we conduct  numerical experiments  to validate  the theoretical results.

Throughout the paper, $C$ and $c$ denote generic positive constants that may depend on $\aone$, $\atwo$, $\mu,$ $b$ and $a$,  but are independent of the  mesh  size $h$ and the time step  $\tau$.

\section{Regularity results} \label{sec:notation}
\se

To express the regularity properties of the solution of problem \eqref{main}, we introduce the Hilbert space 
$\dot H^r(\Omega)\subset L^2(\Omega)$ induced by the norm 
$$
\|v\|_{\dot H^r(\Omega)}^2 = \|A^{r/2}v\|^2=\sum_{j=1}^\infty \lambda_j^r (v,\phi_j)^2,
$$
where $(\cdot,\cdot)$ is the inner product on $L^2(\Omega)$, $\|\cdot\|$ is the induced norm, and
 $\{\lambda_j\}_{j=1}^\infty$ and $\{\phi_j\}_{j=1}^\infty$ are, respectively,  the Dirichlet eigenvalues and 
eigenfunctions of  $A:=-\Delta$, with $\{\phi_j\}_{j=1}^\infty$ being an  orthonormal basis in $L^2(\Omega)$.
Then, for $r$ a nonnegative integer, we have
$
\dot{H}^r(\Omega)=\{\chi \in H^r(\Omega);\, \cL^j\chi=0 \text{ on } \partial \Omega, \text{ for } j<r/2\}$, 
see \cite[Lemma 3.1]{thomee1997}. In particular, $\|v\|_{\dot H^0(\Omega)}=\|v\|$ is the norm on $L^2(\Omega)$, 
$\|v\|_{\dot H^1(\Omega)}=\|\nabla v\|$ is also the norm on $H_0^1(\Omega)$ and $\|v\|_{\dot H^2(\Omega)}=\|A v\|$ is the equivalent 
norm in $H^2(\Omega)\cap H^1_0(\Omega)$. Note that $\{\dot H^r(\Omega)\}$, $r\geq 0$,  form a Hilbert scale of interpolation spaces.

We shall now  derive an integral representation of the solution $u$ of \eqref{main} and  describe the smoothing properties of the solution operator. For a given  $\theta\in (\pi/2,\pi)$,  we denote by $\Sigma_{\theta}$ the sector
$\{z\in \mathbb{C}, \,z\neq 0,\, |\arg z|< \theta\}.$
Since $A$ is selfadjoint and positive definite, its resolvant
$(z I+A)^{-1}:L^2(\Omega)\to L^2(\Omega)$ satisfies 
\begin{equation}\label{res0}
\|(z I+A)^{-1}\|\leq M_{\theta} |z|^{-1} \quad \forall z\in \Sigma_\theta,
\end{equation}
where $M_{{\theta}}=1/\text{dist}(\Sigma_\theta,-1)=1/\sin(\pi-\theta)$.
Here, and in what follows,  we use the same notation $\|\cdot\|$ to denote the operator norm from $L^2(\Omega)\to L^2(\Omega)$.
 We shall employ the  Laplace transform 
$\hat{u}:=\mathcal{L}(u)$ defined by
$ \hat{u}(x,z)=\int_0^\infty e^{-zt}u(x,t)\,dt, $
and its inverse  
$u(x,t)=\frac{1}{2\pi i}\int_{{\mathcal C}} e^{zt}\hat{u}(x,z)\,dz$, 
where the contour of integration ${\mathcal C}$ is any line in the right-half plane parallel to the imaginary axis  and with Im$(z)$ increasing. By  noting that 
$$\widehat{\partial_t^\aone u_t}(z)=z^\aone \hat{u}_t(z)-(I^{1-\aone}u_t)(0),$$
and using the 
 condition $(I^{1-\aone}u_t)(0)=0$,
an application of the Laplace transform to (\ref{a1}) yields
$$(1+a z^{\aone})\hat{u}_t(z)=-\mu ( 1+b z^{\atwo})A\hat{u}(z)+\hat f(z).$$
Substituting $\hat{u}_t(z)$ by $z\hat{u}(z)-v$, we find that
\begin{equation*}
\left((z+a z^{\aone+1})I+\mu ( 1+b z^{\atwo})A\right)\hat{u}(z)=(1+a z^\aone)v+\hat f(z),
\end{equation*}
or
$$\mu ( 1+b z^{\atwo})\left(\dfrac{z+a z^{\aone+1}}{\mu ( 1+b z^{\atwo})}I+A\right)\hat u(z)=(1+a z^\aone)v+\hat f(z).$$
Then, we formally have
\begin{equation}\label{m1-0}
\hat{u}(z)= \hat{E}(z)\left(v+\dfrac{1}{1+a z^{\aone}}\hat{f}\right), 
\end{equation}
where 
\begin{equation}\label{m2}
 \hat{E}(z) := \dfrac{g(z)}{z}\left(g(z)I+A\right)^{-1},
\end{equation}
and $g(z) := \mu^{-1}(z+a z^{\aone+1})/( 1+b z^{\atwo})$. Note that we can recast problem \eqref{main} with $f=0$ to a Volterra integral equation of the form
\begin{equation}\label{volterra}
u(x,t)=v-\int_0^t k(t-\tau)Au(x,\tau)\, d\tau.
\end{equation}
Indeed, by applying the Laplace transform to \eqref{volterra}, we get
\begin{equation}\label{Lvolterra}
\hat{u}(z)=\frac{1}{z}(1+\hat{k}(z)A)^{-1}v.
\end{equation}
Comparing \eqref{Lvolterra} to \eqref{m1-0}, we obtain by the uniqueness of the Laplace transform that equation \eqref{main} is equivalent to \eqref{volterra} with the kernel $k(t)$ satisfying $\hat{k}(z)=g(z)^{-1}.$ In the next lemma, we state a basic property of the function $g(z)$.
\begin{lemma}\label{lem:g(z)}
Let $\theta \in (\pi/2,\frac{\pi}{1+\aone})$ be fixed and set $\bar{\theta} =(1+\aone)\theta<\pi$. Then, for any $z\in  \Sigma_\theta$, $g(z) \in \Sigma_{\bar{\theta}}$  and
\begin{equation}\label{g1} 
|g(z)|\leq M_\theta \frac{1}{\mu} (|z|+a|z|^{\aone +1}) 
\min\left\{1,\frac{1}{b}|z|^{-\atwo}  \right\}.
\end{equation}
%
\end{lemma}
\begin{proof} 
Let $z\in \Sigma_\theta$, i.e.,  $z=re^{i\phi}$ with $|\phi|<\theta$ and  $r>0$. Then, 
$$g(z)= \dfrac{r e^{i\phi} +a r^{1+\aone}e^{i(1+\aone)\phi}+b r^{1+\atwo}e^{i(1-\atwo)\phi}+a b r^{1+\aone +\atwo}e^{i(1+\aone-\atwo)\phi}}{\mu\left((1+b r^{\atwo}\cos (\atwo \phi))^2+b^2 r^{2\atwo}\sin^2 (\atwo \phi)\right)},$$
which shows that  $g(z)\in \Sigma_{\bar{\theta}}$ since  the four terms in the numerator are in $\Sigma_{\bar{\theta}}$ and their imaginary parts have the same sign. 
To prove \eqref{g1},  we note that
$$
|g(z)|= \dfrac{|z+a z^{\aone+1}|}{\mu | 1+b z^{\atwo}|}\leq \dfrac{|z+a z^{\aone+1}|}{\mu \text{dist}(-1,b z^{\atwo})}
\leq\frac{ M_{{\theta}}}{\mu} \left(|z|+a |z|^{\aone+1}\right).
$$ 
We also have
$$
|g(z)|= \frac{|z+a z^{\aone+1}|} {\mu b |z|^{\atwo} \text{dist}(-1, b^{-1}z^{-\beta} )}\leq
\frac{ M_{{\theta}}}{b\mu} \left(|z|^{1-\atwo}+a |z|^{1+\aone-\atwo}\right).
$$
This completes the proof of the lemma.
\end{proof}

Now, using the resolvent estimate (\ref{res0}), we have 
\begin{equation}\label{res1}
\|(g(z) I+A)^{-1}\|\leq M_{\bar{\theta}} |g(z)|^{-1} \quad \forall z\in \Sigma_\theta.
\end{equation}
Then, from the definition of $\hat E(z)$ in (\ref{m2}), we get
\begin{equation}\label{m3} 
\|\hat{E}(z)\|\leq  M_{\bar\theta} |z|^{-1}\quad \forall z\in \Sigma_\theta.
\end{equation}
Hence, by means of  the inverse Laplace transform,  we deduce that the solution operator is given by
\begin{equation}\label{mm}
E(t)=\frac{1}{2\pi i}\int_{\Gamma_{\theta,\delta}} e^{zt}\hat{E}(z)\,dz,
\end{equation}
where,  for  $\theta\in (\pi/2,\pi/(1+\alpha))$ and $\delta>0$, 
$\Gamma_{\theta,\delta}:=\{\rho e^{\pm i\theta}:\rho\geq \delta\}\cup\{\delta e^{i\psi}: |\psi|\leq \theta\}$
is the contour  oriented with an increasing imaginary part. 
Hence, by \eqref{m3} and \cite[Theorem 2.1 and Corollary 2.4]{integral}, we conclude that, 
for the case  $f=0$ and $v\; \in \; L^2(\Omega)$, there exists a unique solution $u$ to \eqref{main} given by $u(t)=E(t)v$ and satisfying 
$$u\in\;C([0,T];\; L^2(\Omega))\cap C(\left(0,T\right];\; \dot{H}^2(\Omega)).$$
To get more details on the smoothing properties of the solution operator,  we first state some properties for the operator $\hat{E}(z)$.
\begin{lemma}\label{lem:Ah} The following estimates hold:
\begin{equation}\label{00}
  \|A\hat{E}(z)\chi\|\leq C_{\theta} |z|^{-1}|g(z)|^{(2-p)/2} \|\chi\|_{\dot H^p(\Omega)} \quad \forall z\in\Sigma_\theta,\quad 0\leq p\leq 2, 
\end{equation} 
\begin{equation}\label{11}
\|A^\nu \hat{E}(z)\chi\|\leq C_{\theta} |z|^{-1}|g(z)|^{\nu}  \|\chi\| \quad \forall z\in\Sigma_\theta,
\quad 0\leq \nu\leq 1,
\end{equation} 
 where $C_{\theta}$ depends only on $\theta$.
\end{lemma}
\begin{proof} Since $A$ commutes with $\hat E(z)$ on $\chi\in \dot H^2(\Omega)$,  we have by \eqref{m3},
\begin{equation}\label{res-3}
\|\hat E(z)A\chi\|\leq M_{\bar\theta} |z|^{-1} \|A\chi\|\quad \forall \chi\in \dot H^2(\Omega).
\end{equation}
On the other hand, by noting that
\begin{equation*}
g(z)\hat{E}(z)=\frac{g(z)}{z}I-A\hat{E}(z),
\end{equation*}
that is
\begin{equation}\label{res-E}
A\hat E(z) = \frac{g(z)}{z}I-g(z)\hat{E}(z),
\end{equation}
we conclude that
\begin{equation}\label{res-5}
\|A\hat E(z)\chi\|\leq  (1+M_{\bar\theta}) |z|^{-1}|g(z)| \|\chi\|\quad \forall \chi\in  L^2(\Omega).
\end{equation}
The estimate \eqref{00} follows by  interpolating \eqref{res-3} and \eqref{res-5} for $p\in [0,2]$. 
The second estimate \eqref{11} follows by interpolating (\ref{m3}) and \eqref{res-5} for $\nu\in [0,1]$, which completes the proof.
\end{proof}


Based on Lemma \ref{lem:Ah}, stability and smoothing properties  of the solution operator $E(t)$ are established in the following Theorem.

\begin{theorem}\label{th:Reg-u}
The following estimates hold for $t\in(0,T]$ and $\nu=0,\,1$:
\begin{equation}\label{reg-u-1}
\|A^\nu E^{(m)}(t)v\| \leq C t^{-m -\nu(\aone-\atwo +1)}\|v\|,\quad v\in L^2(\Omega),\, m\geq 0, 
\end{equation}
\begin{equation}\label{reg-u-2}
\|A^\nu E^{(m)}(t)v\| \leq C t^{-m+(1-\nu)(\aone-\atwo+1)}\|Av\|,\quad v\in \dot H^2(\Omega),\, \nu+m\geq 1, 
\end{equation}
where $C$ is a constant depending on  $\aone$, $\atwo$, $\mu$, $a$, $b$ and $T$.
\end{theorem}
\begin{proof}
We  differentiate both sides of \eqref{mm} with respect to $t$ and then apply the operator $A^\nu$ so that
\begin{equation*}
A^\nu E^{(m)}(t)v=\frac{1}{2\pi i}\int_{\Gamma_{\theta,\delta}} z^me^{zt}A^\nu\hat{E}(z)v\,dz.
\end{equation*}
Then, by \eqref{11}, 
\begin{equation}\label{mm-2}
\|A^\nu E^{(m)}(t)\|\leq C_\theta \int_{\Gamma_{\theta,\delta}} |z|^m e^{Re(z)t} 
|z|^{-1}|g(z)|^{\nu}\,|dz|.
\end{equation}
To evaluate the integral in \eqref{mm-2}, we choose $\delta=1/t$ and set $z=re^{i\varphi}$.
Then, for $\nu=0$, we deduce that
$$
\|A^\nu E^{(m)}(t)\|\leq C_\theta \int_\Gamma |z|^{m-1} e^{Re(z)t} |dz|\leq C_\theta t^{-m},
$$
where here and throughout the paper  $\Gamma:=\Gamma_{\theta,1/t}$.
For the case $\nu=1$, we obtain
\begin{eqnarray*}\|A E^{(m)}(t)\|&\leq & C_\theta \displaystyle\int_{\Gamma} |z|^{m-1} e^{Re(z)t} \left(  |z|+a|z|^{\aone +1}\right) 
\min\left\{1,\frac{1}{b}|z|^{-\atwo}\right\}\,|dz|\\
			&\leq & C_{\theta} t^{-m}(t^{-1}+a t^{-\aone -1})\min\left\{1,\frac{t^{\atwo}}{b}\right\}\\
			&\leq & C_{\theta,T} t^{-m}(t^{-1}+a t^{-\aone -1})\frac{t^{\atwo}}{b}\\
			&\leq & C t^{-m-\aone-1+\atwo},
\end{eqnarray*}
where the last inequality holds since $t^{-1} \leq T^\aone t^{-\aone -1}$.
To prove the estimate \eqref{reg-u-2}, we note that by \eqref{res-E},
$$
 E^{(m)}(t)v= \frac{1}{2\pi i}\int_\Gamma z^me^{zt}({z}^{-1}I-g(z)^{-1}A\hat{E}(z))v\,dz.
$$
As $\int_\Gamma  e^{zt}z^{m-1}\,dz =0$ for $m\geq 1$, and 
$$\|g(z)^{-1}\hat{E}(z))\|\leq M_\theta \mu \min\left\{|z|^{-2}+b|z|^{\atwo -2},\frac{1}{a}(|z|^{-2-\aone}+b|z|^{-2-\aone+\atwo})\right\},$$  we conclude that
\begin{eqnarray*}
\|E^{(m)}(t)v\|&\leq & C \displaystyle\int_\Gamma |z|^me^{Re(z)t} \mu \min\left\{|z|^{-2}+b|z|^{\atwo -2},\frac{1}{a}(|z|^{-2-\aone}+b|z|^{-2-\aone+\atwo})\right\} \|A v\|\,|dz|\\
        &\leq & C t^{-m}\min\left\{t+b t^{1-\atwo },\frac{1}{a}(t^{\aone+1}+bt^{\aone-\atwo+1})\right\}\|A v\|\\
				&\leq & C t^{-m}(t+b t^{1-\atwo })\min\left\{1,\frac{t^\aone}{a}\right\}\|A v\|\\
				&\leq & C t^{-m}(t+b t^{1-\atwo })\frac{t^\aone}{a}\|A v\| \\
				&\leq & C t^{-m+1-\atwo+\aone}\|A v\|,
\end{eqnarray*}				
which shows \eqref{reg-u-2} for $\nu =0$. Finally,  \eqref{reg-u-2} with $\nu =1$ can be obtained from \eqref{reg-u-1} by setting $\nu=0$ and replacing $v$ by $Av$. 
\end{proof}




\section {The spatially semidiscrete problem}\label{sec:FE}
\se
In this section, we describe the Galerkin FE procedure in space and derive optimal error estimates with respect to the smoothness of the solution expressed through the initial data $v$ and the right-hand side $f$.
Let $\cT_h$ be a shape regular  and quasi-uniform triangulation of
the domain $\bar\Omega$ into triangles  $K,$
and let $h=\max_{K\in \cT_h}h_{K},$ where $h_{K}$ denotes the diameter  of $K.$  
The approximate solution $u_h$ of the Galerkin FEM will be sought in the finite element space $V_h$ of continuous piecewise linear functions over the triangulation $\cT_h$:
$$V_h=\{v_h\in C^0(\overline {\Omega})\;:\;v_h|_{K}\;\mbox{is linear for all}~ K\in \cT_h\; \mbox{and} \; v_h|_{\partial \Omega}=0\}.$$
The semidiscrete Galerkin FEM for problem (\ref{main}) is to seek $u_h:(0,T]\to V_h$ such that
\begin{equation} \label{semi-FE}
\left( (1+a \partial_t^\aone)u_{ht} ,\chi\right)  + \bold{a}\left( \mu\left(1 +b \partial_t^\atwo\right) u_h,\chi\right) =  (f,\chi)\quad
\forall \chi\in V_h,\quad t\in (0,T], \quad u_h(0)=v_h,
\end{equation}
with $v_h$  is an appropriate approximation to the initial data $v$ in $V_h$, and $\bold{a}(v,w):= (\nabla v, \nabla w)$ 
is the bilinear form associated with the operator $A$. 
On the space $V_h$, we define the $L^2$-projection $P_h:L^2(\Omega)\rightarrow V_h$ and the Ritz projection
$R_h:H_0^1(\Omega)\rightarrow V_h$,  respectively, by
\begin{equation*}
( P_h\varphi,\chi)=( \varphi,\chi)  \quad \forall \chi\in V_h,
\end{equation*}
\begin{equation*}
 \bold{a}( R_h\varphi, \chi)=\bold{a}( \varphi,\chi)  \quad \forall \chi\in V_h.
\end{equation*}
The operators $P_h$ and $R_h$ satisfy the following approximation properties, see 
\cite{Ciarlet-2002} and \cite{thomee1997}.
\begin{lemma}\label{PR} 
The operators $P_h$ and $R_h$ satisfy
\begin{equation} \label{P_h}
\|P_h\psi -\psi\|+h\|\nabla(P_h\psi -\psi)\| \leq ch^q 
\|\psi\|_{\dot H^q(\Omega)}\quad  \forall \psi\in {\dot H^q(\Omega)},\; q=1,2,
\end{equation}
\begin{equation} \label{R_h}
\|R_h\psi -\psi\|+h\|\nabla(R_h\psi -\psi)\| \leq ch^q 
\|\psi\|_{\dot H^q(\Omega)}\quad  \forall \psi\in {\dot H^q(\Omega)},\; q=1,2.
\end{equation}
In particular, \eqref{P_h} indicates that $P_h$ is stable in $\dot H^1(\Omega)$.
\end{lemma}

We next  introduce the discrete  operator $A_h:V_h\rightarrow V_h$ defined by
\begin{equation*}
(A_h\psi,\chi)=(\nabla \psi,\nabla \chi)  \quad \forall \psi,\chi\in V_h,
\end{equation*}
so that the semidiscrete scheme (\ref{semi-FE}) can be rewritten as
\begin{equation} \label{FE-semi-P}
(1+a \partial_t^\aone)u_{ht} +\mu\left(1 +b \partial_t^\atwo\right) A_h u_h(t)= P_h f(t),\quad\quad t\in (0,T],
\end{equation}
with $u_h(0)=v_h$ and $(I^{1-\alpha}u_{ht})(0)=0$.

\subsection {The homogeneous problem}
Following the  analysis in Section  \ref{sec:notation}, the solution of the homogeneous semidiscrete  problem (\ref{FE-semi-P}) can be represented by
\begin{equation*} \label{u_h}
u_h(t)=\frac{1}{2\pi i}\int_\Gamma e^{zt}\hat E_h(z)v_h\,dz=:E_h(t)v_h,
\end{equation*}
where $\hat  E_h(z)=z^{-1}g(z)(g(z) I+A_h)^{-1}$. Since $A_h$ is selfadjoint and positive definite on $V_h$,  
the estimate in  Lemma \ref{lem:Ah} are also valid for $\hat E_h$.
\begin{lemma}\label{lem:Ah2} 
With $\chi\in V_h$, the following estimates hold:
\begin{equation*}
 \|A_h\hat{E}_h(z)\chi\|\leq C_{\theta} |z|^{-1}|g(z)|^{(2-p)/2} \|{A}_h^{p/2}\chi\| \quad \forall z\in\Sigma_\theta,\quad 0\leq p\leq 2, 
 \end{equation*} 
 \begin{equation*}
\|A_h^\nu \hat{E}_h(z)\chi\|\leq C_{\theta} |z|^{-1}|g(z)|^{\nu}  \|\chi\| \quad \forall z\in\Sigma_\theta,
\quad 0\leq \nu\leq 1. 
 \end{equation*}  
 where $C_\theta$ is independent of the mesh size $h$.
\end{lemma}

By  choosing $v_h=P_hv$, the error of the FE approximation $e_h(t):=u_h(t)-u(t)$  at time $t$ can be represented by 
\begin{equation} \label{error-FE}
e_h(t)=\frac{1}{2\pi i}\int_\Gamma e^{zt}(\hat E_h(z)P_h-\hat E(z))v\,dz= \frac{1}{2\pi i}\int_\Gamma e^{zt}
\frac{g(z)}{z} S_h(z)v\,dz,
\end{equation}
where $S_h(z)=(g(z)I+A_h)^{-1}P_h-(g(z)I+A)^{-1}$. 
In the next lemma, we state important properties of the operator $S_h(z)$ which will play  a key role in the error analysis.

\begin{lemma}\label{G} The following estimate holds for all $z\in\Sigma_\theta$,
\begin{equation} \label{000}
\|S_h(z)v\|+h\|\nabla S_h(z)v\| \leq ch^2 \|v\|,
\end{equation}
\begin{equation} \label{111}
\|S_h(z)v\|_{L^\infty(\Omega)} \leq ch^2 l_h^2 \|v\|_{L^\infty(\Omega)},
\end{equation}
with $l_h=|\ln h|$.
\end{lemma}

The first estimate is given in  \cite{FS-2002,LST-1996,EJLZ2016}. The second one can be found in 
\cite{McLeanThomee2010}. Now we are in position to prove  a nonsmooth data error estimate for the semidiscrete problem.

\begin{theorem}\label{nonsmooth-FE}
Let $u$ and $u_h$ be the solutions of problems \eqref{main} and \eqref{FE-semi-P} with $v\in L^2(\Omega)$ and 
$v_h=P_hv$, respectively, and $f= 0$. Then, for $t>0$, 
\begin{equation} \label{01}
\|e_h(t)\|+h\|\nabla e_h(t)\| \leq C h^2 t^{\atwo-\aone-1}\|v\|.
\end{equation}
\end{theorem}
\begin{proof}
The $L^2$-estimate of the error  follows from the representation \eqref{error-FE} and the estimate \eqref{000}. Indeed, we have
\begin{eqnarray*}
\|e_h(t)\|&\leq & Ch^2\|v\|\displaystyle\int_\Gamma e^{Re(z)t}\frac{|g(z)|}{|z|}|dz|\\
&\leq & C\left( t^{-1+\atwo}+at^{-1-\aone+\atwo}\right)h^2\|v\|\\
    &\leq &C t^{\atwo-\aone-1}h^2\|v\|.
\end{eqnarray*}
The $H^1(\Omega)$-error estimate is established analogously.
\end{proof}

In the next theorem, an error bound is obtained for smooth initial data $v\in \dot H^2(\Omega)$.
\begin{theorem}\label{smooth-FE}
Let $u$ and $u_h$ be the solutions of problems \eqref{main} and \eqref{FE-semi-P} with $v\in \dot H^2(\Omega)$ and 
$v_h=R_hv$, respectively, and $f= 0$. Then, for $t>0$, 
\begin{equation} \label{02}
\| e_h(t)\|+h\|\nabla e_h(t)\| \leq ch^2 \|v\|_{\dot H^2(\Omega)}.
\end{equation}
\end{theorem}
\begin{proof} By  \eqref{res-E} and the identity $A_hR_h = P_hA$, we obtain
\begin{equation*} \label{error-FE-n}
e_h(t)=\frac{1}{2\pi i}\left(-\int_\Gamma e^{zt} z^{-1} S_h(z)Av\,dz
+ \int_\Gamma e^{zt} z^{-1} (R_hv-v)\,dz\right).
\end{equation*}
Then, using \eqref{000} and \eqref{R_h},  we deduce that
\begin{eqnarray*}
\|e_h(t)\|&\leq& ch^2\|Av\|\int_\Gamma e^{Re(z)t}|z|^{-1}|dz|\\
&\leq& ch^2\|Av\|.
\end{eqnarray*}
The $H^1(\Omega)$-error estimate is derived by a similar argument.
\end{proof}

Finally, we derive an error estimate in the maximum norm  by considering \eqref{111} and following the argument in the proof of Theorem \ref{nonsmooth-FE}.
%
\begin{theorem}\label{nonsmooth-FE-infty}
Let $u$ and $u_h$ be the solutions of problems \eqref{main} and \eqref{FE-semi-P} with $v\in L^\infty(\Omega)$ and 
$v_h=P_hv$, respectively, and $f= 0$. Then, for $t>0$, 
\begin{equation} \label{03}
\|e_h(t)\|_{L^\infty(\Omega)}\leq ch^2l_h^2 t^{\atwo-\aone-1} \|v\|_{L^\infty(\Omega)}.
\end{equation}
\end{theorem}

\begin{remark} In Theorem $\ref{smooth-FE}$, the estimate \eqref{02} is still valid  if one chooses the approximation  $v_h=P_hv$, see  the arguments in \cite[Remark 3.1]{EJLZ2016}. Then, by interpolation 
 with \eqref{01}, we obtain for $q\in [0,2]$, 
\begin{equation} \label{02-p}
\|e_h(t)\|+h\|\nabla e_h(t)\| \leq 
ch^2t^{(\atwo-\aone-1)(2-q)/2} \|v\|_{\dot H^q(\Omega)}.
\end{equation}
\end{remark}


\subsection {The inhomogeneous problem}\label{sec:inh}

We now turn back to the inhomogeneous problem \eqref{main}. With a vanishing initial data $v$, 
Laplace transforms yield $\hat u(z) =\mu^{-1} (1+bz^\atwo)^{-1}\left(g(z)I+A\right)^{-1} \hat f(z)$.
 Then, with $\hat f(z)$ being analytic in the sector $\Sigma_\theta$, we can follow the analysis presented for the previous subsection  to obtain, after noting that $|(1+bz^\atwo)^{-1}|\leq M_\theta b^{-1} |z|^{-\atwo}$, the following error estimates:
\begin{equation} \label{01-inh}
\| e_h(t)\|+h\|\nabla e_h(t)\| \leq ch^2 t^{\atwo -1}\|\hat f(z)\|_{L^2(\Omega),\Gamma}
\end{equation}
and 
\begin{equation} \label{02-inh}
\| e_h(t)\|_{L^\infty(\Omega)}\leq ch^2l_h^2t^{\atwo-1}\|\hat f(z)\|_{L^\infty(\Omega),\Gamma},
\end{equation}
where $\Gamma=\Gamma_{\theta,1/t}$ and
$$
\|\hat f(z)\|_{\mathcal{B},\Gamma}:=\sup_{z\in \Gamma} \|\hat f(z)\|_{\mathcal{B}}.
$$
Note that it is possible to remove the term $t^{\atwo -1}$ from the right-hand sides  of \eqref{01-inh} and \eqref{02-inh} if 
$\|z^{1-\atwo} \hat f(z)\|_{L^2(\Omega),\Gamma}$ and $\|z^{1-\atwo}\hat f(z)\|_{L^\infty(\Omega),\Gamma}$ are  finite. A serious restriction in  this approach is that it requires the Laplace transform $\hat f(z)$ to exist, to be analytic in the whole sector $\Sigma_\theta$ and to be such that the norms indicated above are finite.
The simple case with $f(t)=e^t$ shows that the Laplace transform $\hat f(z)=(z-1)^{-1}$ is not analytic in $\Sigma_\theta$ and so the method is not applicable.

We shall now follow a different approach and  prove results showing a classical-type nonsmooth data error estimate that does not use $\hat f(z)$.
%
\begin{theorem}\label{nonsmooth-FE-n}
Let $u$ and $u_h$ be the solutions of problems \eqref{main} and \eqref{FE-semi-P}, respectively. 
Let $q\in[0,2]$ and $p>1/\atwo$. Then, for $t>0$, the following error estimates hold:

(a) If $v\in \dot H^q(\Omega)$ and  $f\in L^p(0,T;L^2(\Omega))$, then
\begin{equation} \label{01-00}
\begin{split}
\|e_h(t)\|+h\|\nabla e_h(t)\| \leq &  
ch^2\left(t^{(\atwo-\aone -1)(2-q)/2}\|v\|_{\dot H^q(\Omega)}+t^{\atwo-1/{ p}}\|f\|_{ L^p(0,T;L^2(\Omega))}\right).
\end{split}
\end{equation}

(b) If  $v\in L^\infty(\Omega)$ and $f\in L^p(0,T;L^\infty(\Omega))$, then
\begin{equation} \label{02-00}
\|e_h(t)\|_{L^\infty(\Omega)}\leq 
ch^2l_h^2 \left(t^{\atwo-\aone -1}\|v\|_{ L^\infty}+t^{\atwo-1/{p}}\|f\|_{ L^p(0,T;L^\infty(\Omega))}\right).
\end{equation}
 \end{theorem}
\begin{proof} Set $G_h(t)= E_h(t)P_h -  E(t)$ and $\bar {G}_h(t)=H_h(t)P_h-H(t)$ where 
\\$H(t)=\mathcal{L}^{-1}\left\{\dfrac{1}{\mu (1+bz^\atwo)}\left(g(z)I+A\right)^{-1}\right\}$. Then, by applying the Laplace transform to  \eqref{a1} and \eqref{FE-semi-P}, 
we  represent the error by 
\begin{equation*}
e_h(t)= G_h(t)v+\int_0^t\bar{G}_h(t-s){f}(s)\,ds:=I+II.
\end{equation*}
The first term $I$ is already bounded in  \eqref{02-p}. For the second term, we apply Lemma \ref{G} to get
\begin{equation*}\label{mm-10}
\|\bar{G}_h(t)\|=\left\|\frac{1}{2\pi i}\int_\Gamma e^{zt}\dfrac{1}{\mu (1+bz^\atwo)}S_h(z)\,dz\right\|\leq ch^2\min\left\{t^{-1},t^{\atwo-1}\right\}\leq c_Th^2t^{\atwo-1}.
\end{equation*}
Using this bound and   H\"older's inequality, it follows that 
\begin{eqnarray*}
\|II\|&\leq & ch^2  \int_0^t (t-s)^{\atwo-1}\|f(s)\| \,ds\\
&\leq & ch^2 \left(\frac{t^{(\atwo-1)\bar{p} +1}}{(\atwo-1)\bar{p} +1}\right)^{1/{\bar{p}}}\|f\|_{ L^p(0,T;L^2(\Omega))},
\end{eqnarray*}
where $\bar p$ is the conjugate exponent of  $p$. This proves \eqref{01-00}.
For the $L^\infty(\Omega)$-estimate,  we first notice the bound
$$\|\bar{G}_h(t)\chi\|_{L^\infty(\Omega)}\leq ch^2l_h^2 t^{\atwo-1} \|\chi\|_{L^\infty(\Omega)},$$ 
and, as a consequence, we obtain
$$
\|II\|_{L^\infty(\Omega)}\leq ch^2l_h^2  \left(\frac{t^{(\atwo-1)\bar{p} +1}}{(\atwo-1)\bar
{p} +1}\right)^{1/{\bar{p}}}\|f(s)\|_{ L^p(0,T;L^\infty(\Omega))}\,ds.
$$
This bound with (\ref{03}) yields the desired  estimate (\ref{02-00}).  
\end{proof}

\section {Time discretization}\label{sec:discrete}
\se
Now, we consider the time discretization of the semidiscrete problem (\ref{FE-semi-P}) 
by convolution quadratures generated by the backward Euler and the second-order backward difference methods.
We  divide the time interval $[0,T]$ into a uniform grid with a time step size $\tau=T/N$, $N$ a positive integer, and set $t_j=j\tau$.  Let $K(z)$ be the Laplace transform of a distribution $k(t)$ on the real line, which vanishes for $t<0$, has its singular support empty or concentrated at $t=0$ and which is analytic for $t>0$.
With $\partial_t$ denoting time differentiation, we then define $K(\partial_t)$ as the operator 
of (distributional) convolution with the kernel $k$: $K(\partial_t)\varphi=k\ast \varphi$ for a 
function $\varphi(t)$ with suitable smoothness. The convolution quadrature developed in \cite{LST-1996,Lubich-2006} and  initiated in \cite{Lubich-1986,Lubich-1988}  refers to an approximation
of $K(\partial_t)\varphi$ by a discrete convolution 
$K(\partial_\tau)\varphi$ at $t=t_n$  as
$$
K(\partial_\tau)\varphi(t_n) = \sum_{j=0}^n \omega_{n-j}(\tau) \varphi(t_j),
$$
where the quadrature weights $\{\omega_j(\tau)\}_{j=0}^{\infty}$ are determined by the generating 
power series
$$
\sum_{j=0}^\infty  \omega_j(\tau) \xi^j=K(\delta(\xi)/\tau)
$$
with $\delta(\xi)$ being a rational function, chosen as the quotient of the generating polynomials of 
a stable and consistent linear multistep method. We shall consider the Backward Euler (BE) 
and the second-order backward difference (SBD) methods, for which $\delta(\xi)=1-\xi$ and $\delta(\xi)=(1-\xi)+ (1-\xi)^2/2$, respectively.
%
%
For the BE method,  the convolution quadrature formula for approximating the fractional integral $\partial_t^{-\alpha}\varphi$ is given by 
$$
 \partial_\tau^{-\alpha}\varphi(t_n) = \sum_{j=0}^n  \omega_{n-j} \varphi(t_j), \text{ where } 
\sum_{j=0}^\infty  \omega_{j} \xi^j=[(1-\xi)/\tau]^{-\alpha}, \quad \omega_j=\tau^\alpha(-1)^{j}
\left(\begin{array}{c}
-\alpha\\
j
\end{array}\right),
$$
while for the SBD method, the quadrature weights are provided by the formula \cite{Lubich-1986}:
$$
\omega_j=\tau^\alpha(-1)^{j}\left(\frac{2}{3}\right)^\alpha\sum_{l=0}^j
3^{-l}
\left(\begin{array}{c}-\alpha\\j-l\end{array}\right)
\left(\begin{array}{c}-\alpha\\l\end{array}\right).
$$
An important property of the convolution quadrature is that it maintains some relations of the continuous convolution. 
For instance, the associativity of convolution is valid for the convolution 
quadrature  \cite{Lubich-2004} such as
\begin{equation}\label{s1}
K_1( \partial_\tau) K_2( \partial_\tau)=K_1 K_2( \partial_\tau)\quad 
\text{ and }\quad 
K_1( \partial_\tau)(k\ast\varphi)=(K_1(\partial_\tau)k)\ast\varphi.
\end{equation}

The approximation properties of the convolution quadrature, which represent the key  ingredient in the error analysis, are described in the following lemma, see \cite[Theorem 4.1]{Lubich-1988} and \cite[Theorem 2.2]{Lubich-2004}.

\begin{lemma}\label{lem:Lubich}
Let $K(z)$ be  analytic in the sector $\Sigma_\theta$ and satisfies
\begin{equation*}\label{s3}
\|K(z)\|\leq M|z|^{-\mu}\quad \forall z\in \Sigma_\theta,
\end{equation*}
for some real $\mu$ and $M$. Assume that the linear multistep method is strongly $A$-stable and of 
order $p\geq 1$. Then, for $\varphi(t)=ct^{\nu-1}$, the convolution quadrature satisfies
\begin{equation*}
\|K(\partial_t)\varphi(t) - K( \partial_\tau)\varphi(t)\|  \leq \left\{
\begin{array}{ll}
C t^{\mu-1+\nu-p} \tau^p, & \nu\geq p,\\
C t^{\mu-1} \tau^\nu, & 0< \nu\leq p.
\end{array} \right.
\end{equation*}
\end{lemma}


\subsection{Error analysis for the BE method}
We start by considering the time discretization of  ($\ref{FE-semi-P}$) by the convolution 
quadrature based on the BE method. To this end, we  integrate (\ref{FE-semi-P}) from $0$ to $t$  
and use the initial condition  $(I^{1-\alpha}u_{ht})(0)=0$ to get
\begin{equation*}
\left( u_h-v_h+a\left(\partial_t^{\aone-1} u_{ht}\right)\right)   +\mu\left( \partial_t^{-1}+b \partial_t^{\atwo -1}\right)  A_h u_h=\partial_t^{-1} f_h,
\end{equation*}
where $f_h= P_h f$.
This can be simplified as 
\begin{equation*}
\left( u_h-v_h+a\left(\partial_t^{\aone} u_h-\partial_t^{\aone} v_h\right) \right)  +\mu\left( \partial_t^{-1}+b \partial_t^{\atwo -1}\right)  A_h u_h=\partial_t^{-1} f_h,
\end{equation*}
and after rearrangements, we get
\begin{equation}\label{s5}
\left( 1+a\partial_t^{\aone}\right)  u_h  +\mu\left( \partial_t^{-1}+b \partial_t^{\atwo -1}\right)  A_h u_h=(1+a\partial_t^{\aone} )v_h+\partial_t^{-1} f_h.
\end{equation}
 %
The convolution terms in \eqref{s5} can then be approximated by convolution quadratures generated  by the BE method. Hence,  the fully discrete  solution $U_h^n$ satisfies at $t_n=n\tau$, 
\begin{equation}\label{s6}
\left( 1+a\partial_\tau^{\aone}\right)  U_h^{n}  +\mu\left( \partial_\tau^{-1}+b \partial_\tau^{\atwo -1}\right)  A_h U_h^{n}=(1+a\partial_\tau^{\aone} )v_h+\partial_\tau^{-1} f_h.
\end{equation}
If we apply $\partial_\tau$ to \eqref{s6} and use the associativity of convolution in \eqref{s1}, we arrive at the following scheme: with  $U_h^0=v_h$, find $U_h^n$ for $n=1,\cdots,N$ such that
\begin{equation*}
\left( \partial_\tau+a\partial_\tau^{1+\aone}\right)  U_h^{n}  +\mu\left( 1+b \partial_\tau^{\atwo }\right)  A_h U_h^{n}=a\partial_\tau^{1+\aone} v_h+ f_h^n,
\end{equation*}
where $f_h^n=f_h(t_n)$. 
In view of \eqref{s5} and \eqref{s6}, we have for the homogeneous problem, 
$$ u_h = G(\partial_t)v_h\qquad  \mbox{and}\qquad  U^n_h = G(\partial_\tau)v_h,$$ where
$G(z)=g(z)(g(z)I+ A_h)^{-1}$ and $g(z)$ is the function defined in Lemma \ref{lem:g(z)}. Then, we may represent the   error  $U^n_h-{u}_h(t_n)$ at $t=t_n$  by 
\begin{equation}\label{s7-k}
U^n_h-{u}_h(t_n)= \left( G(\partial_\tau)- G(\partial_t)\right)v_h.
\end{equation}
If we let   $ G_1(z)=- g^{-1}(z)  G(z)$ and noting that
\begin{equation}\label{id-5}
G(z)=I-g^{-1}(z) G(z)A_h,
\end{equation}
we deduce that
\begin{equation}\label{s8}
U^n_h-u_h(t_n)=\left(G_1(\partial_\tau) -  G_1(\partial_t)\right)A_hv_h.
\end{equation}
Then, by  Lemma \ref{lem:Lubich}, we obtain the following error estimates for smooth and nonsmooth initial data.
%
\begin{lemma}\label{lem:BE}
Let ${u}_h$ and $U^n_h$ be the solutions of problems $(\ref{s5})$ and $(\ref{s6})$, respectively, with 
$f=0$.  Then, the following estimates hold:

(a) If $v\in \dot{H}^2(\Omega)$ and $v_h =R_hv$, then
\begin{equation}\label{s10}
\|U^n_h-u_h(t_n)\|\leq C \tau (t_n^{\aone}+b t_n^{\aone-\atwo})\|v\|_{\dot{H}^2(\Omega)}.
\end{equation}

(b) If $v\in L^2(\Omega)$ and $v_h =P_hv$, then
\begin{equation}\label{s9}
\|U^n_h-u_h(t_n)\|\leq C \tau t_n^{-1}\|v\|.
\end{equation}

\end{lemma}
\begin{proof} We recall that, by \eqref{res1}, 
\begin{equation}\label{s00}
\|G(z)\|\leq M_\theta\quad  \forall z\in \Sigma_\theta.
\end{equation}
Then, we have
$$\| G_1(z)\|\leq M_{\theta}|g(z)|^{-1}\leq  M_{\theta}\mu\left( |z|^{-1}+b|z|^{\atwo-1}\right) \min\left\lbrace {1,\frac{|z|^{-\aone}}{a}}\right\rbrace .$$
Applying  Lemma \ref{lem:Lubich} to \eqref{s8} with $\nu=1$, $p=1$ and  $\mu=\aone+1, 1+\aone-\atwo$, respectively, 
we deduce that
$$
\|U^n_h-u_h(t_n)\|\leq C \tau \frac{\mu}{a}\left( t_n^{\aone}+b t_n^{\aone-\atwo}\right)\| A_hv_h\|.
$$
As $v_h=R_h v$, we use the identity  $ A_hR_h= P_hA$ so that
$$
\| A_h v_h\| =\| A_h R_h v\| = \| P_h A v\| \leq C \|A v\|=C \|v\|_{\dot{H}^2(\Omega)},
$$
where the last inequality follows from the $L^2(\Omega)$-stability of $ P_h$. This shows \eqref{s10}.

To derive \eqref{s9}, we use \eqref{s00} and 
 apply Lemma \ref{lem:Lubich} to \eqref{s7-k} with $\mu=0$, $\nu=1$ and $p=1$,  to obtain
$$
\|U^n_h-{u}_h(t_n)\|\leq C \tau t_n^{-1}\|v_h\|.
$$
The estimate  follows then by  the $L^2(\Omega)$-stability of $P_h$, which completes the proof.
\end{proof}


Now, recalling the estimates derived  in Theorems \ref{nonsmooth-FE} and \ref{smooth-FE} for the semidiscrete problem, 
we summarize our results in the next theorem as follows. 

\begin{theorem}\label{thm:BE}
Let $u$ and $U^n_h$ be the solutions of problems $(\ref{main})$ and $(\ref{s6})$, respectively, with $f=0$. Then, the following error estimates hold:

(a) If $v\in \dot{H}^2(\Omega)$ and $v_h =R_hv$, then
\begin{equation*}
\|U^n_h-u(t_n)\|\leq C (h^2+\tau t_n^{\aone}+\tau b t_n^{\aone-\atwo})\|v\|_{\dot H^2(\Omega)}.
\end{equation*}

(b) If $v\in L^2(\Omega)$ and $v_h =P_hv$, then
\begin{equation*}
\|U^n_h-u(t_n)\|\leq  C (h^2t_n^{\atwo-\aone-1}+\tau t_n^{-1})\|v\|.
\end{equation*}
\end{theorem}

\begin{remark}\label{rem:BE-2} Note that, from \eqref{s7-k},
$$
\nabla(U^n_h-{u}_h(t_n))= \left( (\nabla G)(\partial_\tau)- (\nabla G)(\partial_t)\right)v_h.
$$
Since 
$
\|\nabla G(z)\| \leq   c |g(z)|^{1/2}\leq c |z|^{(1+\aone-\atwo)/2}, 
$
an application of Lemma \ref{lem:Lubich} with $\nu=1$ and  $\mu=-(1+\aone-\atwo)/2$ yields
$$
\|\nabla(U^n_h-{u}_h(t_n))\|\leq c \tau t_n^{(-1-\aone+\atwo)/2-1}\|v\|.
$$
Similarly, we  obtain for smooth initial data
$$
\|\nabla(U^n_h-{u}_h(t_n))\|\leq c \tau t_n^{(-1+\aone-\atwo)/2}\|v\|_{\dot H^2(\Omega)}.
$$
Thus, by interpolation, it follows that 
\begin{equation*}
\|\nabla(U^n_h-{u}_h(t_n))\|\leq c \tau t_n^{-1}\|v\|_{\dot H^1(\Omega)}.
\end{equation*}
\end{remark}


We now derive error estimates for the inhomogeneous problem with $v=0$.

\begin{theorem}\label{thm:BE-n}
Let $u$ be the solution of the problem $(\ref{main})$  with $v=0$ and $f\in L^{\infty}(0,T;L^2(\Omega))$, and let
$U^n_h$  be the solution of  $(\ref{s6})$ with $v_h=0$. Then, the following error estimate holds:
\begin{equation}\label{s12-n}
\|U^n_h-u(t_n)\|\leq  c 
\left(h^2 t_n^{\atwo}\|f\|_{L^\infty(0,T;L^2(\Omega))}+\tau t_n^{\aone}\|f(0)\|+\tau\int_0^{t_n} (t_n-s)^{\aone}\|f'(s)\|ds\right).
\end{equation}
\end{theorem}

\begin{proof}  Taking into account the estimate derived in Theorem \ref{nonsmooth-FE-n} for $v=0$, it suffices to bound $U^n_h-u_h(t_n)$.
In view of \eqref{s5} and \eqref{s6},  the error is represented  by
\begin{equation*}\label{s7-kn}
U^n_h-{u}_h(t_n)= (  F_h(\partial_\tau)-  F_h(\partial_t))f_h,
\end{equation*}
where $ F_h(z)=\dfrac{1}{\mu(1+bz^{\atwo})}(g(z)I+ A_h)^{-1}$. Using the expansion $f_h(t)=f_h(0)+(1\ast f_h')(t)$ and the 
second relation in \eqref{s1}, we have
\begin{equation*}\label{s7-knn}
U^n_h-{u}_h(t_n)= ( F_h(\partial_\tau)-  F_h(\partial_t))f_h(0)+(( F_h(\partial_\tau)-  F_h(\partial_t))1)\ast  f_h'(t_n)=:I+II.
\end{equation*}
Then, by Lemma \ref{lem:Lubich} (with $\mu=1+\aone$ and $\nu=1$) and the $L^2(\Omega)$-stability of $ P_h$, we obtain
$$
\|I\|\leq c\tau t_n^{\aone} \|f_h(0)\|\leq c\tau  t_n^{\aone}\|f(0)\|.
$$
For the second term, Lemma \ref{lem:Lubich} yields
$$
\|II\|\leq  \int_0^{t_n} \|(( F_h(\partial_\tau)-  F_h(\partial_t))1)(t_n-s) f_h'(s)\| \leq 
c\tau \int_0^{t_n} (t_n-s)^{\aone} \|f_h'(s)\|ds.
$$
Together, these estimates with \eqref{01-00} obtained for $v=0$ and $f\in L^\infty(0,T;L^2(\Omega))$ complete the proof of \eqref{s12-n}.
\end{proof}


\subsection{Error analysis for the SBD method}
Now, we consider the time discretization of  the semidiscrete problem ($\ref{FE-semi-P}$) by  convolution quadrature based on the second-order backward difference formula, and study convergence rates  for smooth and nonsmooth initial data. Recall that the semidiscrete solution $ u_h$ 
is given by
$$ 
 u_h=G(\partial_t)\left(v_h+(1+a\partial_t^{\aone})^{-1}\partial_t^{-1}f_h\right).
$$
From the estimates in Lemma \ref{lem:Lubich}, it is clear that  a second-order error bound cannot be achieved,  
if for instance, $\varphi$ is constant (i.e., $\nu=1$), even if a high-order multistep is used. To maintain the 
second-order time accuracy, we modify the scheme following the strategy proposed in \cite{LST-1996,Lubich-2006,EJLZ2016}. To do so, we use  \eqref{id-5} and the splitting 
$f_h=f_h(0)+\tilde f_h$, where $\tilde f_h= f_h-f_h(0)$,  so that $u_h$ can be represented by
\begin{equation}\label{k6n-0}
 u_h= v_h+G(\partial_t)\left(-g^{-1}(\partial_t) A_hv_h
+(1+a\partial_t^{\aone})^{-1}(\partial_t^{-1} f_h(0) + \partial_t^{-1} \tilde f_{h})\right).
\end{equation}
This leads to the corrected numerical scheme
\begin{equation}\label{k6n}
U_h^n= v_h+G(\partial_\tau)\left(-g^{-1}(\partial_\tau)
\partial_\tau \partial_t^{-1}  A_h v_h +(1+a\partial_\tau^{\aone})^{-1}\left(\partial_t^{-1} f_h(0) + \partial_\tau^{-1} \tilde f_{h}\right)\right),
\end{equation}
where the symbol $\partial_\tau$ refers to time approximation based on the SBD method, and  
the exact contribution $\partial_t^{-1}$ is kept in the formula in order  to preserve the second-order time accuracy.

For  numerical purposes, it is essential to rewrite \eqref{k6n} as a time stepping algorithm. 
Letting $1_\tau=(0,3/2,1,\cdots)$ so that $1_\tau=\partial_\tau\partial_t^{-1}1$ at grid point $t_n$, we
 apply the operator $(I+ g(\partial_\tau)^{-1} A_h)$ to both sides of \eqref{k6n} and use the associativity 
 of convolution to  finally arrive at  the equivalent form  
\begin{equation*}\label{k6n2}
(I+ g(\partial_\tau)^{-1} A_h)(U_h^n-v_h)=
-g(\partial_\tau)^{-1}  A_h  1_\tau v_h 
+(1+a\partial_\tau^{\aone})^{-1}\left(\partial_\tau^{-1} 1_\tau f_h(0) + \partial_\tau^{-1} \tilde f_{h}\right).
\end{equation*}
Applying the operator $\partial_\tau(1+a\partial_\tau^{\aone})$, we obtain 
\begin{equation}\label{SBD}
((\partial_\tau+a\partial_\tau^{1+\aone})I+ \mu(1+b\partial_\tau^{\atwo}) A_h)(U_h^n-v_h)=
-\mu(1+b\partial_\tau^{\atwo}) A_h  1_\tau v_h 
+1_\tau f_h(0) +  \tilde f_{h}.
\end{equation}
Since $1v_h- 1_\tau v_h=(v_h,-1/2v_h,0,\cdots)$, we thus define the time stepping scheme as: with $U^0_h=v_h$, find $U_h^n$ such that 

$$\left(\frac{3}{2}\tau^{-1}+a\partial_\tau^{\aone+1}\right)\left(U_h^1-v_h\right)+\mu \left(1+b \tilde\partial_\tau^{\atwo} \right) A_h U_h^1+\frac{\mu}{2} A_h v_h=f^1_h+\frac{1}{2}f^0_h,$$
and for $n\geq 2$
$$
\partial_\tau U_h^n+a\partial_\tau^{\aone+1}( U_h^n-v_h)+\mu \left(1+ b \tilde\partial_\tau^{\atwo}\right) A_hU_h^n=f^n_h,
$$
where the modified convolution quadrature $\tilde\partial_\tau^{\atwo}$ is given by 
$$
\tilde\partial_\tau^{\atwo}\varphi^n=
\left( \sum_{j=1}^n \omega_{n-j}^{\atwo}\varphi^j+ \frac{1}{2}\omega_{n-1}^{\atwo}\varphi^0\right),
$$
with the weights $\{\omega_j^{\atwo}\}$ being generated by the SBD method.

Now, using Lemma \ref{lem:Lubich}, we derive the following error bounds for the homogeneous problem with smooth and nonsmooth initial data $v$.
\begin{lemma}\label{lem:SBD}
Let ${u}_h$ and $U^n_h$ be the solutions of problems $(\ref{FE-semi-P})$ and $(\ref{SBD})$, respectively, with 
$f=0$.  Then, the following estimates hold:

(a) If $v\in \dot{H}^2(\Omega)$ and $v_h =R_hv$, then
\begin{equation}\label{k10}
\|U^n_h-u_h(t_n)\|\leq C   \left(  \tau^2  t_n^{\aone-1}+\tau^2 t_n^{\aone-\atwo-1}\right) \|v\|_{\dot{H}^2(\Omega)}.
\end{equation}

(b) If $v\in L^2(\Omega)$ and $v_h =P_hv$, then
\begin{equation}\label{k9}
\|U^n_h-u_h(t_n)\|\leq C \tau^2 t_n^{-2}\|v\|.
\end{equation}
\end{lemma}

\begin{proof} Note that, in view of \eqref{id-5}, we can split the error into  
\begin{equation}\label{k9b}
U^n_h-u_h(t_n)= \left(  G_1(\bar\partial_\tau)- G_1(\partial_t)\right) \partial_t^{-1}  A_hv_h,
\end{equation}
where  $G_1(z)=-  zg^{-1}(z) G(z)$. We easily verify that 
$$\| G_1(z)\|  \leq M_{{\theta}}\mu |z| \left(|z|^{-1}+b|z|^{\atwo-1}\right)\min\left\{1,\frac{|z|^{-\aone}}{a}\right\}.$$
Next, we apply Lemma \ref{lem:Lubich} to \eqref{k9b} with $\nu=2$, $p=2$ and  $\mu=\aone,\aone-\atwo$, respectively, 
and  the desired estimate \eqref{k10} follows then by using the identity $ A_hR_h= P_hA$.

For the estimate \eqref{k9}, we write the error as 
\begin{equation}\label{k9a}
U^n_h-u_h(t_n)= \left( G_2(\partial_\tau)- G_2(\partial_t)\right)\partial_t^{-1} v_h,
\end{equation}
where $G_2(z)=z(G(z)-I)$. 
By noting that $\|G_2(z)\|\leq (1+ M_{{\theta}})|z|\;  \forall z\in \Sigma_\theta$, a use of 
\eqref{k9a}, Lemma \ref{lem:Lubich} (with $\mu=-1$, $\nu=2$ and $p=2$) and the $L^2(\Omega)$-stability
of $P_h$ yields the estimate \eqref{k9}.
\end{proof}

Recalling the estimates derived in Theorems \ref{nonsmooth-FE} and \ref{smooth-FE}, we summarize our results 
for the SBD method as follows.

\begin{theorem}\label{thm:SBD}
Let $u$ and $U^n_h$ be the solutions of problems $(\ref{main})$ and $(\ref{SBD})$, respectively, with  
$f=0$. Then, the following error estimates hold:

(a) If $v\in \dot{H}^2(\Omega)$ and $v_h =R_hv$, then
\begin{equation*}\label{k12}
\|U^n_h-u(t_n)\|\leq C (h^2+ \tau^2  t_n^{\aone-1}+ \tau^2  t_n^{\aone-\atwo-1})\|v\|_{\dot{H}^2(\Omega)}.
\end{equation*}

(b) If $v\in L^2(\Omega)$ and $v_h =P_hv$, then
\begin{equation*}\label{k14}
\|U^n_h-u(t_n)\|\leq C (h^2t^{\atwo-\aone-1}_n+\tau^2 t_n^{-2})\|v\|.
\end{equation*}

\end{theorem}
\begin{remark} Following the same analysis as in Remark \ref{rem:BE-2}, we obtain for the SBD method
$$
\|\nabla(U^n_h-{u}_h(t_n))\|\leq c \tau^2 t_n^{(-3-\aone+\atwo)/2-1}\|v\|,
$$
and
$$
\|\nabla(U^n_h-{u}_h(t_n))\|\leq c \tau^2 t_n^{(-3+\aone-\atwo)/2}\|v\|_{\dot H^2(\Omega)}.
$$
Then, by interpolation, it follows that 
\begin{equation*}
\|\nabla(U^n_h-{u}_h(t_n))\|\leq c \tau^2 t_n^{-2}\|v\|_{\dot H^1(\Omega)}.
\end{equation*}
\end{remark}

For the inhomogeneous problem with $v=0$, we have the following  error estimates.

\begin{theorem}\label{thm:SBD-inh}
Let $u$ be the solution of the problem $(\ref{main})$  with $v=0$ and $f\in L^{\infty}(0,T;L^2(\Omega))$, and 
let $U^n_h$ be the solution $(\ref{SBD})$ with $v_h=0$. Then, the following error estimate holds:
\begin{equation*}
\begin{split}
\|U^n_h-u(t_n)\|\leq  c 
\Big( &h^2 t_n^{\atwo}\|f\|_{L^\infty(0,T;L^2(\Omega))}+ \tau^2t_n^{\aone-1} \|f(0)\| \\
 &\left.+\tau^2 t_n^{\aone}\|f'(0)\|+\tau^2\int_0^{t_n}(t_n-s)^{\aone}\|f''(s)\|ds\right).
\end{split}
\end{equation*}
\end{theorem}

\begin{proof} By  Theorem \ref{nonsmooth-FE-n}, it suffices to bound $U^n_h-u_h(t_n)$. Let $\tilde F=zF(z)$. By using the expansion $\tilde f_h= t f_h'+t\ast  f_h''$ in \eqref{k6n-0} and  \eqref{k6n}, we rewrite the
solutions $u_h(t_n)$ and $U_h^n$ as 
\begin{eqnarray*}
  u_h(t_n) &= & \tilde F(\partial_t)t f_h(0)+ F(\partial_t)t f_h'(0) + ( F(\partial_t)t)\ast f_h'',\\
  U_h^n  & = & \tilde F(\partial_\tau)t f_h(0)+ F(\partial_\tau)t f_h'(0) + ( F(\partial_\tau)t)\ast f_h'',
\end{eqnarray*}
respectively. Then, we have
\begin{equation*}
\begin{split}
 U_h^n - u_h(t_n) = &\; (\tilde F(\partial_\tau)-\tilde F(\partial_t))t f_h(0)+( F(\partial_\tau)- F(\partial_t))t f_h'(0)\\
 & \;+ (( F(\partial_\tau)- F(\partial_t))t)\ast f_h''=:I+II.\\
\end{split}
\end{equation*}
For the first term, Lemma \ref{lem:Lubich} with $\mu=\aone$ and $\nu=2$ and the $L^2(\Omega)$-stability of
$ P_h$ give $\|I\|\leq c t_n^{\aone-1}\tau^2 \|f'(0)\|.$
For the second term,  Lemma \ref{lem:Lubich} with $\mu=\aone+1 $ and $\nu=2$ and the $L^2(\Omega)$-stability of
$ P_h$ yield 
\begin{eqnarray*}
\|II\|&\leq & \|( F(\partial_\tau)- F(\partial_t))tf_h'(0)\| + 
\int_0^{t_n} \|(( F(\partial_\tau)- F(\partial_t))t)(t_n-s) f_h''(s)\|ds \\
&\leq &
c\tau^2 \left( t_n^{\aone}\|f_h'(0)\|+\int_0^{t_n} (t_n-s)^{\aone}\|f_h''(s)\|ds\right),
\end{eqnarray*}
which completes the proof. \end{proof}

\section{ Numerical experiments}  \label{sec:NE}
\se


In this section, we conduct numerical experiments to validate the convergence theory presented in Section \ref{sec:discrete}. We provide  two sets of  numerical examples with  zero and nonzero right-hand side data $f$,  defined on the square domain $\Omega=(0,1)^2$.  We examine separately the spatial and temporal convergence rates at a fixed final time $T=0.5$. In  our computation, we fix the parameters $\mu =a=b =1$.
For the homogeneous problem, we consider the following smooth and nonsmooth initial data:
\begin{itemize}
\item[ (a)]  $v(x,y)=xy(1-x)(1-y)\in \dot H^{2}(\Omega)$,  
\item[ (b)] $v(x,y)=\chi_{(0,1/2]\times(0,1)}(x,y)$, where $\chi_S$ is the characteristic function of the set $S$.
Here  $v \in \dot H^{\epsilon}(\Omega)$ for $0\le \epsilon<1/2$. 
\end{itemize}

The exact solution is difficult to obtain for these examples, so we compute a reference solution on a very refined mesh with $h=1/512$ and employ a time step size $\tau=1/500$.

To examine the temporal convergence rates of the proposed BE and SBD schemes, we employ a uniform 
 mesh in time with a time step size $\tau=T/N$, and  choose a sufficiently small  mesh size  $h = 1/512$ so that the error incurred by spatial discretization is negligible.
We measure the error $e^n := u(t_n)-U^n$ by the normalized $L^2(\Omega)$-norm 
$\| e^n\|/\|v\|$. The numerical results are presented in  Table \ref{table:1} for 
cases (a) and (b) with different values of $\alpha$ and $\beta$.

\begin{table}[h]
\begin{center}
\caption{$L^2$-error for cases (a) and (b) with $h=1/512$.}
\label{table:1}
\begin{tabular}{|c|r|cccc|cccc|}
\hline
 &\multicolumn{5}{|c|}{case (a)} & \multicolumn{4}{|c|}{case (b)}\\
\hline
 {$\aone,\, \atwo$} &
 $N$& BE & rate& SBD & rate & BE & rate& SBD & rate \\
\hline
  & 20 & 1.43e-3  &  & 7.69e-5&   & 8.93e-4&    & 4.83e-5 &  \\
 {$\aone=0.25$}& 40 &7.10e-4  & 1.01 & 1.85e-5  & 2.06 &  4.43e-4 &  1.01  &1.16e-5 & 2.06\\
{$\atwo=0.75$} & 80 & 3.54e-4 & 1.01 & 4.46e-6  &2.05  & 2.21e-4 &1.01  &2.80e-6& 2.05\\
& 160 &  1.77e-4 &  1.00 & 1.02e-6 &  2.13 & 1.10e-4 &  1.00  &6.40e-7 &  2.13 \\
 &  320 &  8.82e-5 &  1.00 & 1.67e-7 &  2.61 & 5.50e-5  & 1.00 & 1.05e-7  & 2.61  \\
 
 \hline
 &20 & 2.76e-4 &  & 4.37e-5&  & 1.58e-4&    & 2.51e-5 &   \\

{$\aone=0.5$} & 40 &9.58e-5 & 1.53 & 1.07e-5  &2.03 &  5.50e-5 &  1.53 &6.16e-6 & 2.03\\
{$\atwo=0.5$} & 80 & 3.91e-5 &1.29 & 2.50e-6  &2.10 & 2.24e-5&1.29&1.44e-6& 2.10\\
  & 160 & 1.76e-5   & 1.15 &   5.58e-7 &  2.16  & 1.01e-5 &  1.15&  3.21e-7 &  2.16 \\
  & 320 & 8.37e-6 &  1.07 &  9.03e-8 &  2.63 & 4.80e-6 &  1.07 & 5.18e-8 &  2.63\\
  
\hline
 &20 & 1.78e-2    &  & 8.60e-3          &  & 1.05e-2&    & 8.20e-3 &   \\

{$\aone=0.75$} & 40 & 1.12e-2  & 0.67    & 1.95e-3    &2.14& 7.25e-3    &  0.53 &1.63e-3    &2.34\\
{$\atwo=0.25$} & 80 & 6.26e-3    &0.84 &4.54e-4    &2.10 & 4.37e-3   &0.73   &3.57e-4   & 2.19\\
  & 160 & 3.31e-3      & 0.92   &   1.02e-4    &  2.15& 2.41e-3    &  0.86   &  7.82e-5    &  2.19 \\
  & 320 & 1.70e-3    &  0.96    &  1.66e-5    &  2.62 & 1.26e-3    &  0.93    & 1.26e-5    & 2.63\\
  
\hline 

\end{tabular}
\end{center}
\end{table}

In the table, the {\tt rate} refers to the empirical convergence
rate, when the time step size $\tau$ halves. The numerical results show convergence rates of order $O(\tau)$ and $O(\tau^2)$  for the BE  and  SBD schemes, respectively. We also observe that both schemes exhibit a steady 
convergence for both smooth and nonsmooth data, which  confirms the theoretical convergence rates.


\begin{table}[t]
\begin{center} 
\caption{$L^2$-error for cases (a) and (b) with $\aone=0.25$, $\atwo=0.75$ as $t_N\to 0$, $N=10$.}
\label{table:T1}
\begin{tabular}{|c|c|cccccl|}
\hline
case& meth. &  1e-3 & 1e-4 & 1e-5 & 1e-6 & 1e-7 & rate\\
\hline
(a) & BE  & 5.37e-3   & 2.41e-3     & 8.54e-4  &  2.82e-4    & 9.09e-5     & 0.49  $(0.50)$ \\
   & SBD & 3.99e-4 & 1.47e-4  & 4.85e-5  & 1.57e-5  & 5.01e-6  & 0.49  $(0.50)$\\ 
 \hline
(b) & BE  & 5.38e-3   &  4.32e-3    & 3.21e-3     & 2.45e-3     &  1.85e-3    & 0.12 $(0.125)$ \\
& SBD &  4.99e-4    &  3.63e-4  & 2.78e-4   &  2.11e-4  &  1.59e-4   & 0.12 $(0.125) $\\
\hline
\end{tabular}
\end{center}
\end{table}


To study the temporal error more closely, we investigate the prefactors in Theorems \ref{thm:BE} and \ref{thm:SBD}.
By neglecting the spatial error (i.e., dropping the $O(h^2)$ term) and  taking the number of time steps $N$ as  fixed, the error bounds in  Theorems \ref{thm:BE} show as $t_N\to 0$,
\begin{equation*}
\|U_h^N-u(t_N)\|\leq c  t_N^{\left(\aone-\atwo+1\right)} N^{-1}\|v\|_{\dot H^2(\Omega)},
\end{equation*}
and
\begin{equation*}
\|U_h^N-u(t_N)\|\leq c N^{-1}\|v\|.
\end{equation*}
By interpolating these results, we obtain for $q\in [0,2]$ 
\begin{equation}\label{discrete-1}
\|U_h^N-u(t_N)\|\leq c  t_N^{\left(\aone-\atwo+1\right)q/2} N^{-1}\|v\|_{\dot H^q(\Omega)}.
\end{equation}
Similarly, for the SBD scheme, there holds from Theorem \ref{thm:SBD} as $t_N\to 0$,
\begin{equation}\label{discrete-2}
\|U_h^N-u(t_N)\|\leq c  t_N^{\left(\aone-\atwo+1\right)q/2}N^{-2}\|v\|_{\dot H^q(\Omega)}.
\end{equation}

For fixed $N=10$ and $h=1/512$, we present the computed normalized $L^2$-norm of the error  in Table \ref{table:T1} for cases (a) and (b) as $t_N\to 0$. The predicted rate with respect to $t_N$  computed from 
\eqref{discrete-1} and \eqref{discrete-2} is given between brackets. 
The temporal error should theoretically behaves like $O(t_N^{\left(\aone-\atwo+1\right)q/2})$ for both  the BE and SBD schemes.
From the table, we notice 
that the error decreases like $O(t_N^{1/2})$ in the smooth case (a), whereas  it decreases  like $O(t_N^{1/8})$ in the nonsmooth case (b). The results agree well with the convergence theory.  

%

 

Now, we investigate  the spatial discretization error. To do so, we fix the time step size $\tau=1/500$ and perform the computations using the SBD scheme so that the temporal discretization error is negligible.
In Table \ref{table:2}, we list the normalized $L^2(\Omega)$-norm and the $L^\infty(\Omega)$-norm of the error  for 
 cases (a) and (b). We observe a convergence rate $O(h^2)$ for the $L^2(\Omega)$-norm of the error for smooth and nonsmmoth initial data,  which  confirm the predicted rates. The results also show the validity  of the  convergence rates in the $L^\infty(\Omega)$-norm (ignoring a logarithmic factor).

\begin{table}[t]
\begin{center}
\caption{Error for cases (a) and (b) with $\tau=1/500$.}
\label{table:2}
\begin{tabular}{|c|r|cccc| cccc|}
\hline
&\multicolumn{5}{|c|}{case (a)} & \multicolumn{4}{|c|}{case (b)} \\
\hline
$\aone,\, \atwo$ &$M$& $L^2$-error & rate & $L^\infty$-error & rate & $L^2$-error & rate & $L^\infty$-error & rate\\
\hline
&  8	&2.50e-3 &		    & 1.78e-4       &	     &  1.35e-3  & &  3.92e-3  &\\
$\aone=0.25$&16	& 6.44e-4 &	1.96    & 4.63e-5 &  1.95& 3.47e-4  & 1.96&  1.27e-3  & 1.62\\
$\atwo=0.75$&32	& 1.62e-4 & 1.99	&1.17e-5  & 1.99 & 8.75e-5 &  1.99&  3.91e-4  & 1.70\\
&64	& 4.02e-5&	2.01	&  2.90e-6 &  2.01& 2.18e-5  & 2.00&  1.16e-4 &  1.75\\
&128	&9.67e-6 &	2.06& 7.08e-7  & 2.04& 5.35e-6 &  2.03&  3.34e-5 &  1.79\\
\hline
&     8	      & 1.73e-5 &  &1.16e-6 &  &9.45e-6 & &1.34e-5 &\\
$\aone=0.5$&16	     & 4.99e-6 &  1.80&  3.40e-7  & 1.77 &2.70e-6 & 1.81& 3.90e-6 &  1.78\\
$\atwo=0.5$&32	& 1.29e-6  & 1.95&  8.85e-8 &  1.94 & 6.95e-7 &  1.96& 1.01e-6 &  1.95\\
&64	& 3.22e-7 &  2.00 & 2.21e-8 &  2.00 &1.73e-7& 2.00& 2.53e-7 &  2.00\\
&128	&7.69e-8  & 2.07&5.31e-9&   2.06&4.13e-8& 2.07& 6.07e-8&  2.06\\
\hline
&     8	      & 1.61e-2  &  &1.10e-3 &    &9.59e-3 & & 1.63e-2 &\\
$\aone=0.75$&16	     & 4.14e-3 &  1.96 &  2.88e-4  & 1.93 &2.51e-3  & 1.93 & 4.31e-3  & 1.92\\
$\atwo=0.25$&32	& 1.04e-3  & 1.99  & 7.28e-5 &  1.99& 6.33e-4  & 1.99&  1.08e-3  & 1.99\\
&64	& 2.57e-4 &  2.01 &  1.81e-5  & 2.01 &1.57e-4  & 2.01  & 2.69e-4 &  2.01\\
&128	&6.13e-5  & 2.07 &  4.35e-6 &  2.06 & 3.74e-5  & 2.07  & 6.44e-5 &  2.06\\
\hline
\end{tabular}
\end{center}
\end{table}



\begin{figure}[t!]
   \centering 
  
   \subfigure[t=0]
    {
       \includegraphics[width=6cm]{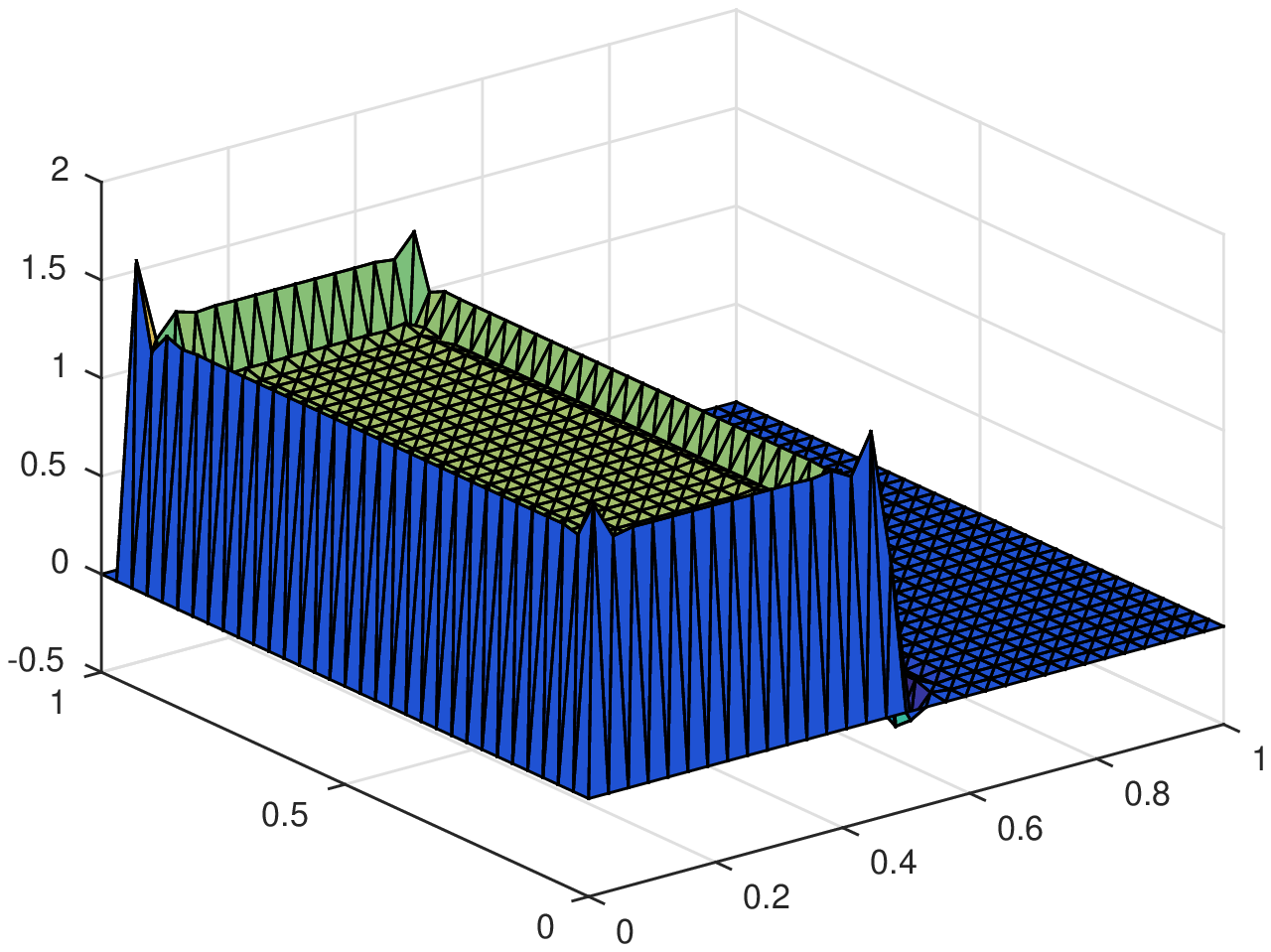}
    }
   \subfigure[t=0.1]
   {
       \includegraphics[width=6cm]{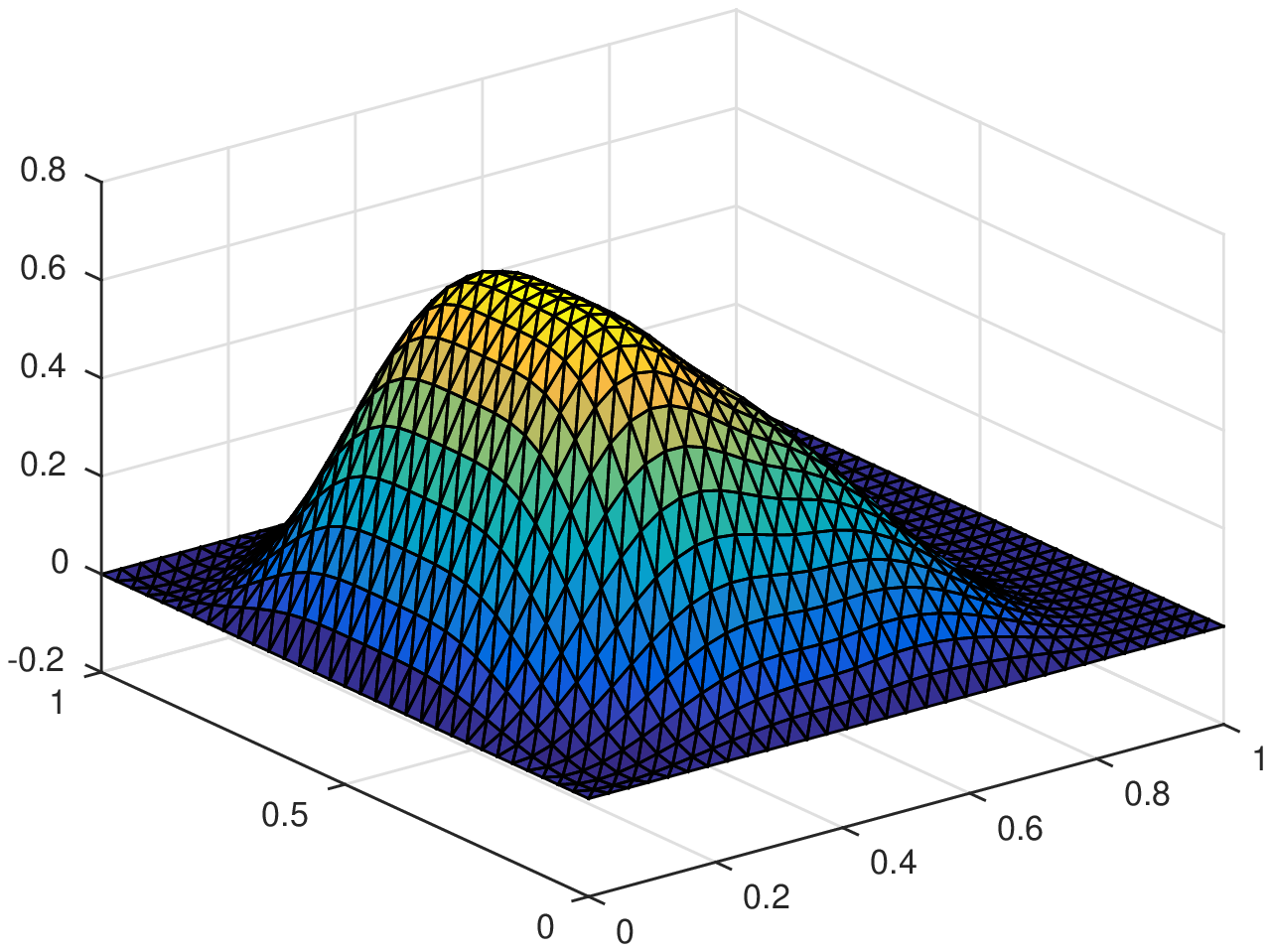}
   }
    \subfigure[t=0.25]
    {
       \includegraphics[width=6cm]{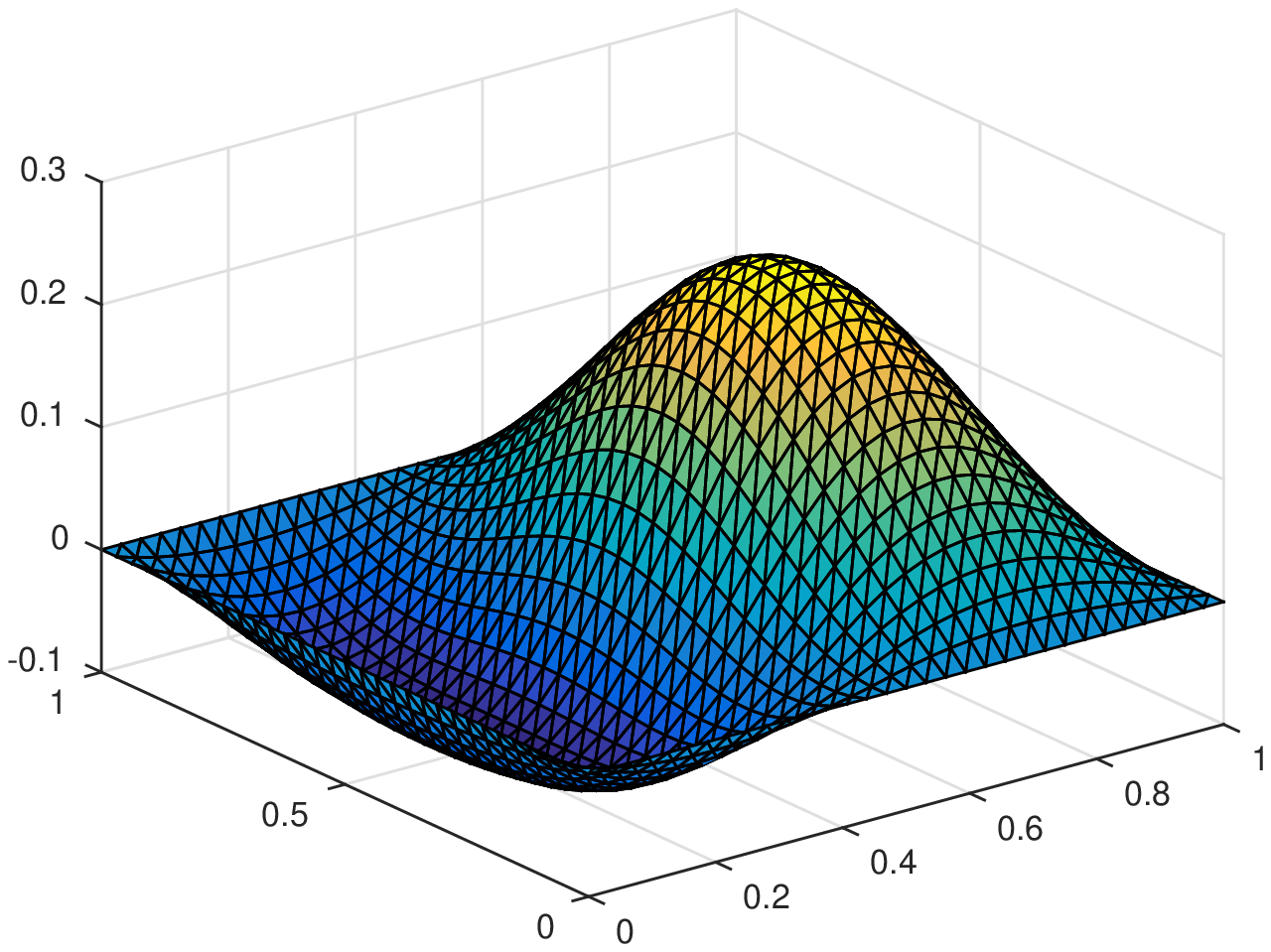}
    }
    \subfigure[t=0.3]
    {
        \includegraphics[width=6cm]{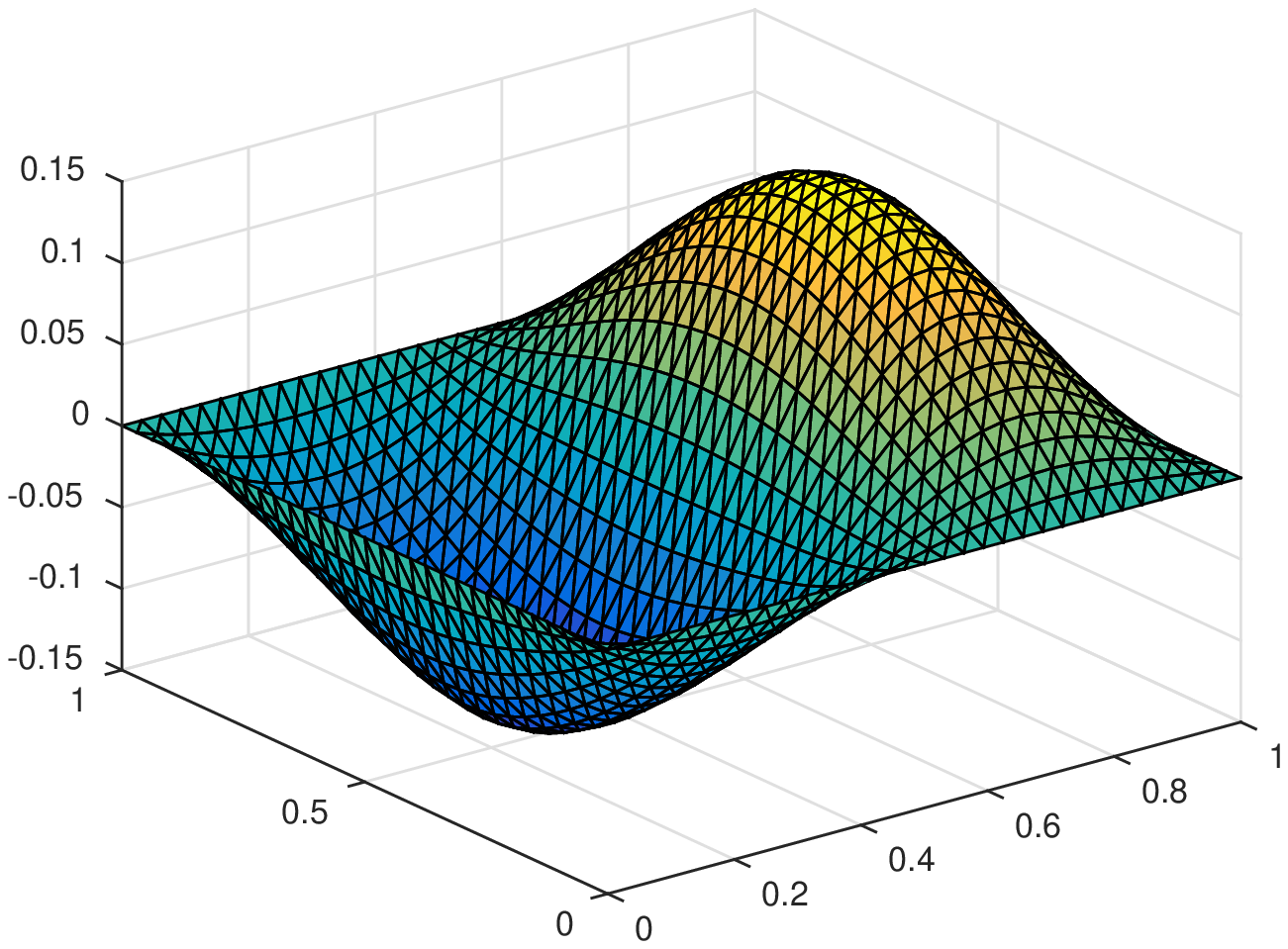}
    }
   
     \caption{Numerical solution for case (b) with $b=0$ and $\aone=0.5$, computed using the SBD scheme with 
     $h=1/32$.}
     \label{Fig:al=0.5,b=0}
     \end{figure}

\begin{figure}[t!]
   \centering

    \subfigure[$\aone=0.25$]
    {
       \includegraphics[width=6cm]{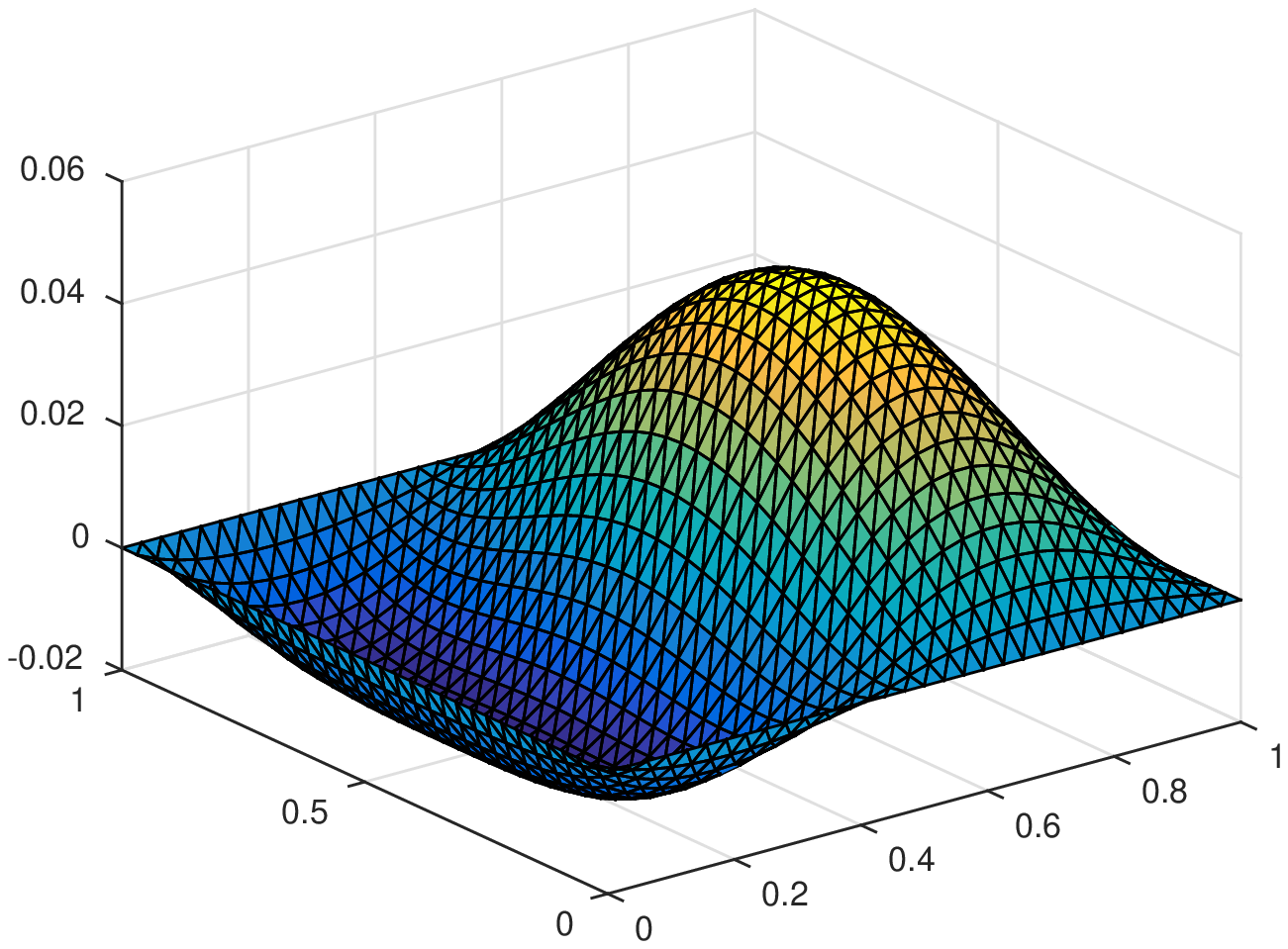}
    }
    \subfigure[$\aone=0.75$]
    {
       \includegraphics[width=6cm]{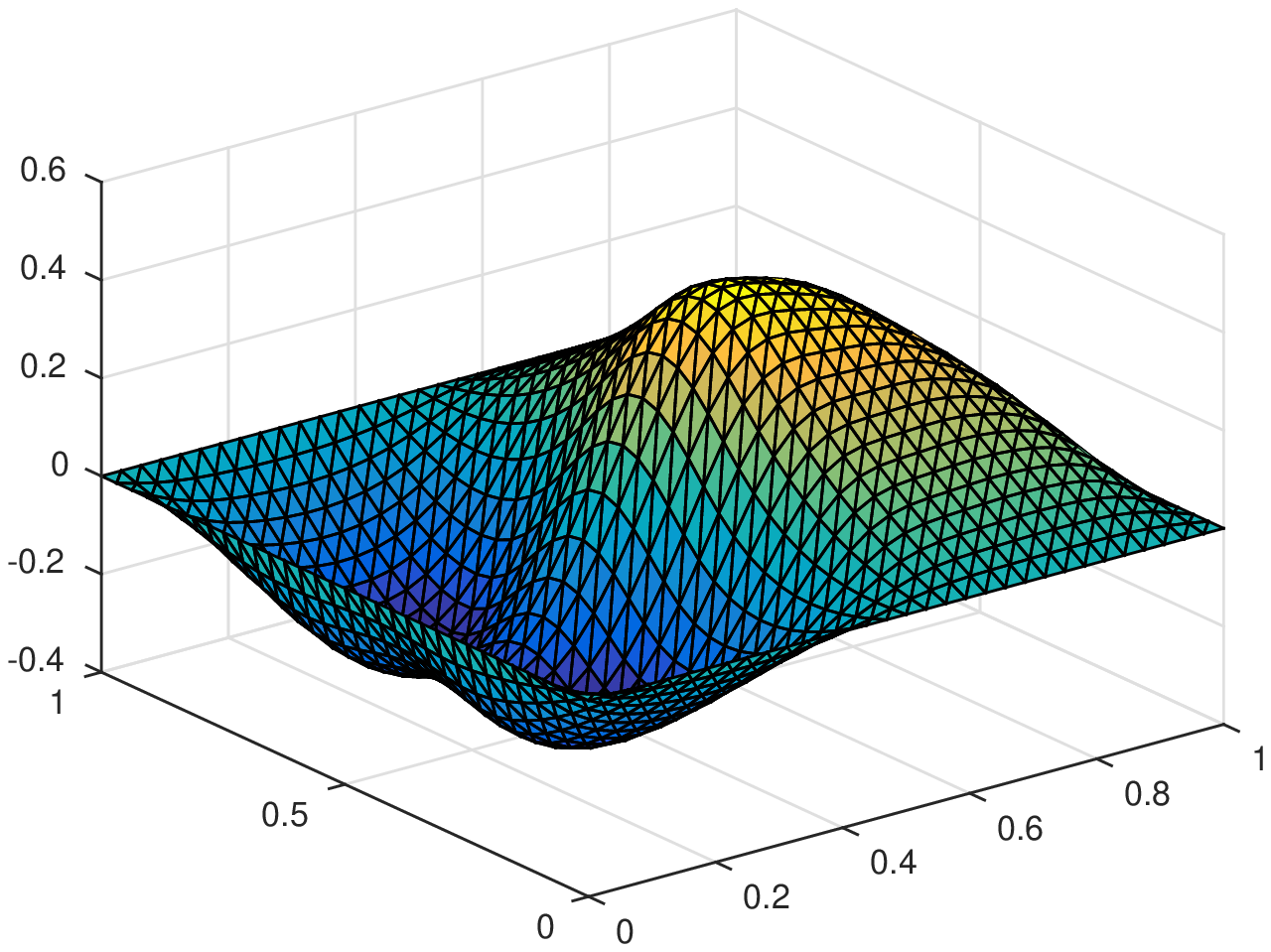}
    }
   
     \caption{Numerical solution for case (b) at $t=0.3$, with  $b=0$ and  different values of $\alpha$,
      computed using the SBD scheme with $h=1/32$.}
     \label{Fig:b=0}
     \end{figure}
     
\begin{figure}[t]
   \centering 
  
    \subfigure[t=0.15]
    {
       \includegraphics[width=6cm]{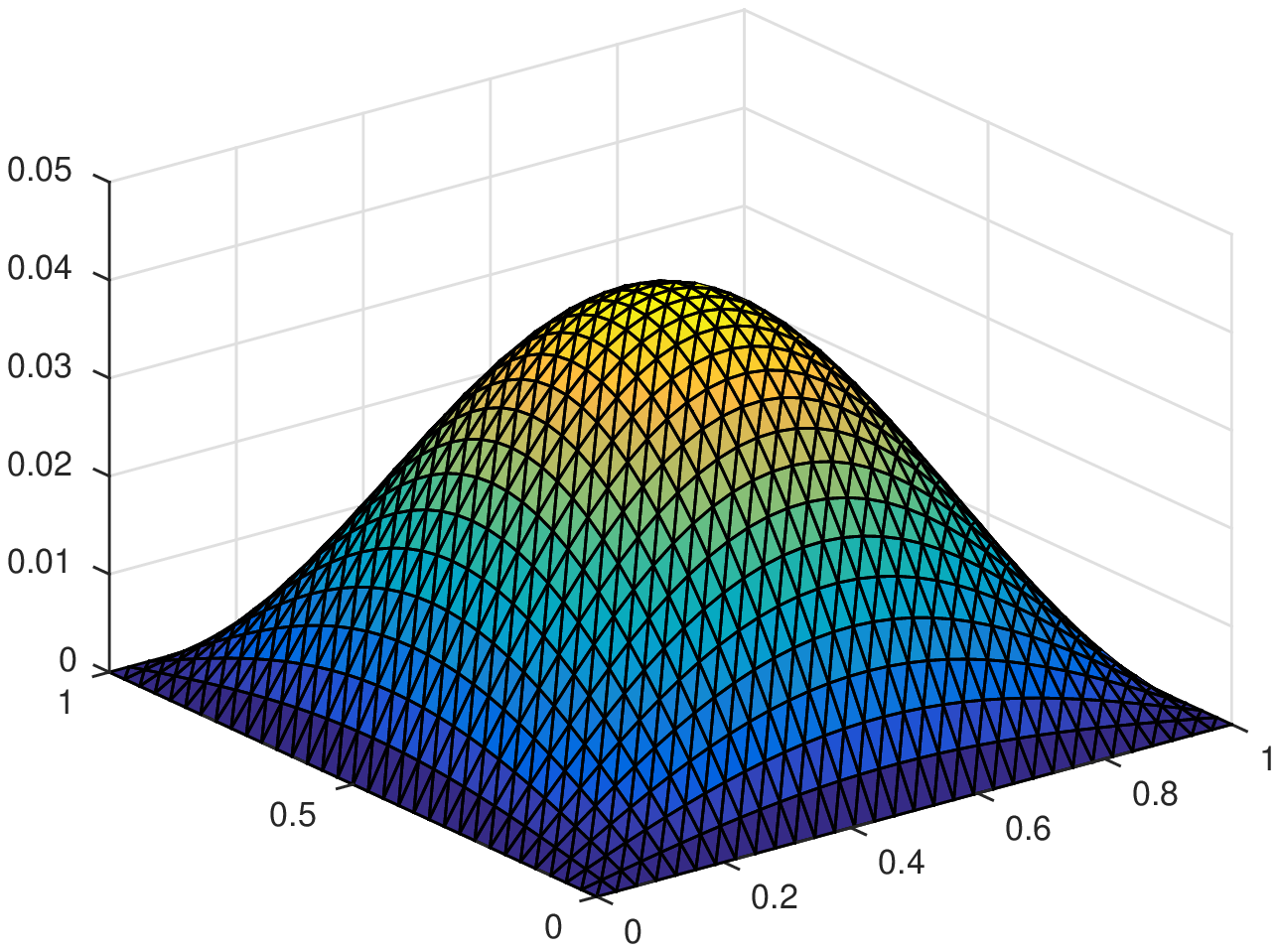} 
    }
    \subfigure[t=0.3]
    {
       \includegraphics[width=6cm]{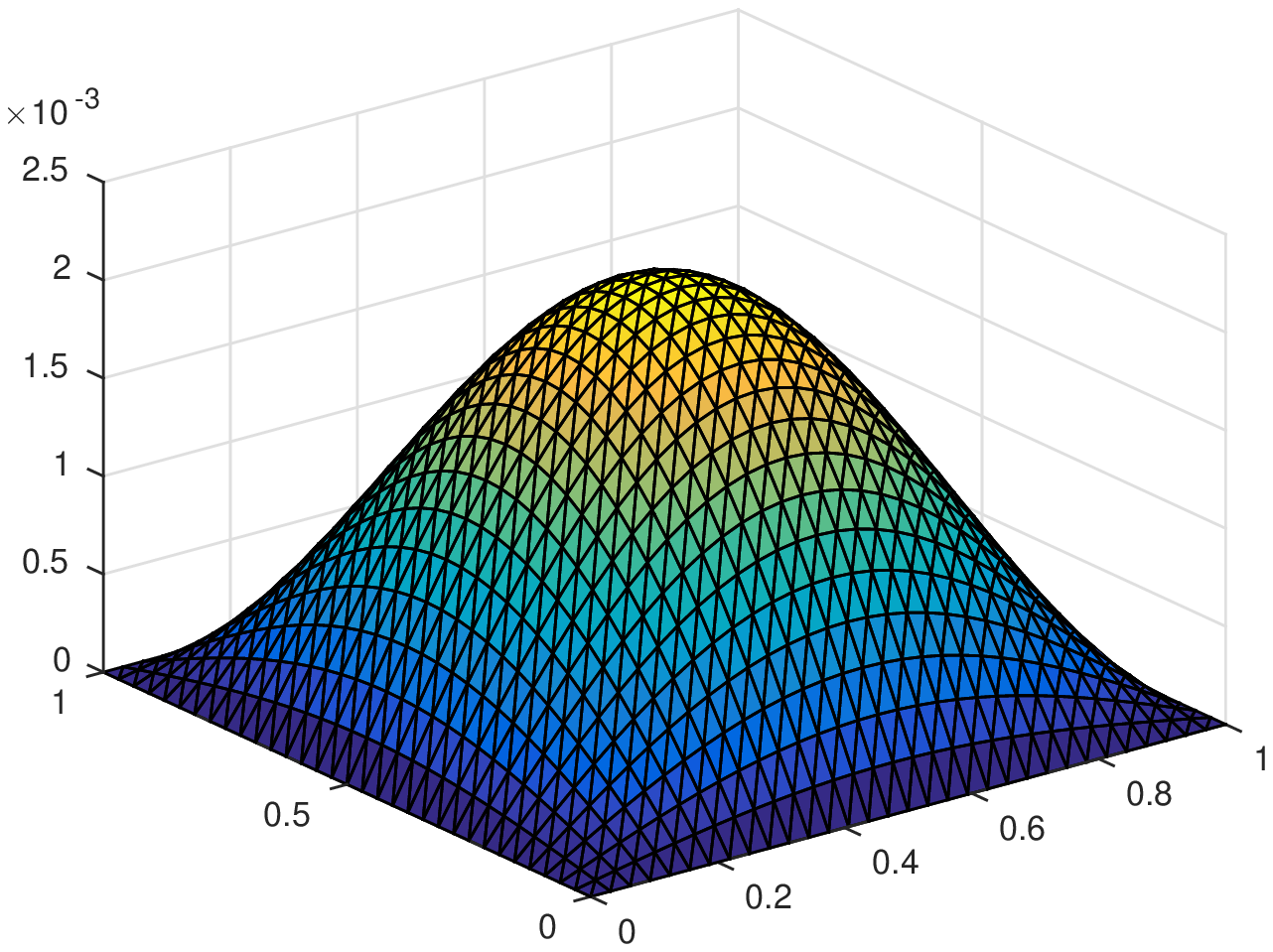}
    }

     \caption{Numerical solution for case (b) with $\aone=\atwo=0.5$ and $b=1$,  
     computed using the SBD scheme with $h=1/32$.} 
     \label{Fig:al=0.5,b=1}
     \end{figure}


 By neglecting the temporal error and fixing $h$, we investigate the spatial prefactors in Theorems \ref{thm:BE} and  \ref{thm:SBD}. 
In Table \ref{table:prefactors}, we report the numerical results  obtained as $t\to 0$ with $h$ being constant.   
The results indicate  that the spatial error essentially stays unchanged in the smooth case, whereas it deteriorates as $t\to 0$ in the nonsmooth case. Since the initial data in case (b) belongs to $\dot H^{\epsilon}(\Omega)$ for $0\leq \epsilon<1/2$, it is expected that the error  grows like $O(t_N^{\left(\atwo-\aone-1\right)(2-q)/2})$. Hence, the empirical convergence rate in Table \ref{table:prefactors} agrees  with the theoretical one, i.e., $3\left(\atwo-\aone-1\right)/4=-0.375$ for $\aone=0.25$ and $\atwo=0.75$.

\begin{table}[h]
\begin{center} 
\caption{$L^2$-error for cases (a) and (b)  with  $\aone=0.25$, $\atwo=0.75$:  $t\to 0$,  $h=1/64$, $N=500$.}
\label{table:prefactors}
\begin{tabular}{|c|c|cccccc|}
\hline
 & method &  1e-3 & 1e-4 & 1e-5 & 1e-6 & 1e-7 & rate\\
\hline
(a) & BE  & 3.46e-4 & 4.41e-4 & 4.86e-4 & 5.04e-4 & 5.12e-4 & -0.01 $(0)$\\ 
    & SBD & 4.37e-4 & 4.94e-4 & 5.08e-4 & 5.12e-4 & 5.15e-4 & -0.00 $(0)$\\ 
  \hline
(b) & BE  & 3.55e-4 & 7.88e-4 & 1.77e-3 & 4.00e-3 & 9.12e-3 & -0.36 $(-0.375)$\\ 
  & SBD &  4.05e-4 & 8.28e-4 & 1.80e-3 & 4.02e-3 & 9.14e-3 & -0.35 $(-0.375)$\\ 
\hline
\end{tabular}
\end{center}
\end{table}
%

\begin{table}[h!]
\begin{center}
\caption{Error for cases (c) and (d) with $\aone=0.25$, $\atwo=0.75$ and $\tau=1/500$.}
\label{table:Non-homogeneous}
\begin{tabular}{|r|cccc|cccc|}
\hline
\multicolumn{5}{|c|}{case (c)} & \multicolumn{4}{|c|}{case (d)} \\
\hline
$M$& $L^2$-error & rate & $L^\infty$-error &  rate & $L^2$-error & rate & $L^\infty$-error & rate\\
\hline
  8    & 3.00e-2 &      & 6.72e-2 &      &  9.43e-4  &       &  4.29e-3 &   \\
  16   & 8.47e-3 & 1.82 & 1.94e-2 & 1.79 & 2.43e-4   &  1.96 &  1.41e-3 &  1.60\\
  32   & 2.18e-3 & 1.96 & 5.02e-3 &	1.95 & 6.13e-5   &  1.99 &  4.38e-4 &  1.69\\
  64   & 5.43e-4 & 2.00 & 1.27e-3 &	1.99 & 1.53e-5   &  2.00 &  1.31e-4 &  1.74\\
  128  & 1.29e-4 & 2.07	& 3.17e-4 &	2.00 & 3.77e-6   &  2.02 &  3.80e-5 &  1.79\\
\hline
\end{tabular}
\end{center}
\end{table}

In Figure \ref{Fig:al=0.5,b=0}, we display the profile of the numerical solution in case (b) when $a=\mu=1$, 
$b=0$ and $\alpha=0.5$,  at different times. We observe that the solution oscillates with a slow decay, which reflects in particular the  wave feature of the model \eqref{main} when $b=0$. The oscillations in the figure are not numerical artifacts. Indeed, they are inherited from the $L^2$-projection $P_hv$ which is oscillatory as shown in Figure \ref{Fig:al=0.5,b=0}(a). Furthermore, as observed from Figure \ref{Fig:b=0}, the solution is diffusive (the oscillations are quickly damped) for small $\aone$ and the oscillations are more pronounced for larger $\aone$, with a slower decay, showing again the wave feature of the model.  In Figure \ref{Fig:al=0.5,b=1}, we display the profile of the numerical solution when $b=a=\mu=1$ and $\alpha=\beta=0.5$. The oscillations  are quickly damped and  not visible as in the previous case. For further details on the behavior of the solution of \eqref{main}, see \cite{B1} for the case with $a= 0$ and \cite{B1,B5} when $a \neq 0$.

Lastly, we consider the inhomogeneous problem \eqref{main} with a zero initial data $v$ and briefly examine the spatial convergence rates established in Section \ref{sec:discrete}. We consider the following cases with smooth and nonsmooth right-hand side data $f$:
\begin{itemize}
\item[ (c)] $v=0$ and $f(x,y,t)=\left(2t+\dfrac{2a t^{1-\aone}}{\Gamma(2-\aone)}+8\pi^2\mu t^2+\dfrac{16\pi^2\mu b t^{2-\atwo}}{\Gamma(3-\atwo)}\right)\sin(2\pi x)\sin(2\pi y)$. 
\item[ (d)] $v=0$ and  $f(x,y,t)= (1+t^{0.2})\chi_{(0,1/2]\times(0,1)}(x,y)$. 
\end{itemize}

For case (c), the  exact solution is given by $u(x,y,t)=t^2\sin(2\pi x)\sin(2\pi y)$, whereas an explicit   
form is not available in case (d). A reference solution  is then computed on the very fine grid.
 The results reported in Table \ref{table:Non-homogeneous} are obtained using the standard Galerkin FEM in space and  the SBD method in time. To validate the error bounds derived in Section \ref{sec:discrete},  we choose a very small time step $\tau=1/500$ and examine the spatial accuracy at the final time $T=0.5$. From the table, a convergence rate  of order $O(h^2)$ is observed, for both smooth and nonsmooth right-hand side data, which clearly confirm our theoretical  results.


\end{document}